\documentclass[11pt, a4paper, reqno, oneside]{amsart}

\usepackage{amsmath, amsopn, amsfonts, amsthm, amssymb, amscd, enumerate, multicol, scalefnt}

\usepackage{amsmath, amsfonts, amssymb, amscd, enumerate, scalefnt}

\usepackage{etex}
\usepackage{hyperref}
\usepackage{t1enc}
\usepackage{graphicx}
\usepackage[all]{xy}
\usepackage{color}
\usepackage{slashed}
\usepackage[mathscr]{euscript}

\usepackage{amscd}
\usepackage{color}
\usepackage{amssymb}
\usepackage{amsmath}
\usepackage{amsfonts}
\usepackage{latexsym}
\usepackage{verbatim}

\newtheorem{definition}{Definition}[section]
\newtheorem{lemma}[definition]{Lemma}
\newtheorem{theorem}[definition]{Theorem}
\newtheorem{proposition}[definition]{Proposition}
\newtheorem{corollary}[definition]{Corollary}
\newtheorem{remark}[definition]{Remark}

\newtheorem{assumption}[definition]{Assumption}

\theoremstyle{plain}
\newtheorem{thmint}{Theorem}

\font\ddpp=msbm10  scaled \magstep 1  %Caracteres "doble palo".
\def\RR{\hbox{\ddpp R}}         %Numeros reales
\def\C{{\hbox{\ddpp C}}}       %Numeros complejos
\def\frg{{\frak g}}

\def\frn{{\frak n}}

\def\Re{{\frak R}{\frak e}\,}

\def\nilm{\Gamma\backslash G}

\def\db{{\bar{\partial}}}

\title[The Anomaly flow on nilmanifolds]
{The Anomaly flow on nilmanifolds}
\author{Mattia Pujia}
\address[M. Pujia]{Dipartimento di Matematica G. Peano, Universit\`a di Torino, V. Carlo Alberto 10, 10123 Torino, Italy}
\email{mattia.pujia@unito.it}

\author{Luis Ugarte}
\address[L. Ugarte]{Departamento de Matem\'aticas\,-\,I.U.M.A.\\
Universidad de Zaragoza\\
Campus Plaza San Francisco\\
50009 Zaragoza, Spain}
\email{ugarte@unizar.es}

\subjclass[2010]{}%Primary 53C44; Secondary 53C15, 53C30, 53C55}

\thanks{The first-named author was partially supported by G.N.S.A.G.A. of I.N.d.A.M. 
The second-named author was supported by the projects MTM2017-85649-P (AEI/FEDER, UE), and E22-20R ``\'Algebra y Geometr\'ia'' (Gobierno de Arag\'on/FEDER)} 

\begin{document}

\begin{abstract}
We study the Anomaly flow on $2$-step nilmanifolds with respect to any Hermitian connection in the Gauduchon line. 
In the case of flat holomorphic bundle, the general solution to the Anomaly flow is given for any initial invariant Hermitian metric. 
The solutions depend on two constants $K_1$ and $K_2$, and we study the qualitative behaviour of the Anomaly flow in terms of their signs, as well as the convergence in Gromov-Hausdorff topology. 
The sign of $K_1$ is related to the conformal invariant introduced by Fu, Wang and Wu. 
%We study the qualitative behaviour of the Anomaly flow, which allows us to find %nilmanifolds admitting both immortal and ancient solutions. 
In the non-flat case, we find the general evolution equations of the Anomaly flow under certain initial assumptions. This allows us to detect non-flat solutions to the Hull-Strominger-Ivanov system on a concrete nilmanifold, which appear as stationary points of the Anomaly flow with respect to the 
Bismut connection.
\end{abstract}

\maketitle

%%% indice
\setcounter{tocdepth}{2} \tableofcontents

\section{Introduction}

\noindent
The {\em Anomaly flow} is a coupled flow of Hermitian metrics introduced by Phong, Picard and Zhang in \cite{PPZ18}. The flow, which was originally proposed as a new tool to detect explicit solutions to the Hull-Strominger system,  leads to some interesting problems in complex non-K\"ahler geometry and its study is just began \cite{FHP19,FP19_2,FP19,PPZ18_2,PPZ18_3,PPZ18_4,PPZ19}.\medskip

Let $X$ be a 3-dimensional complex manifold equipped with a nowhere vanishing $(3,0)$-form $\Psi$ and a holomorphic vector bundle $\pi:E\to X$. 
%%%The Anomaly flow is 
In this paper we consider 
the coupled flow of Hermitian metrics $(\omega_t,H_t)$, with $\omega_t$ on $X$ and $H_t$ along the fibers of $E$, given by
\begin{equation}\label{AF-G}
\begin{aligned}
\partial_t(\Vert\Psi\Vert_{\omega_t}\,\omega_t^2)\,\, = &\,\,\,  {i\partial\overline\partial \omega_t - \frac{\alpha'}{4}\left( {\rm Tr}(Rm_t^\tau\wedge Rm_t^\tau) - {\rm Tr}(A_t^\kappa\wedge A_t^\kappa)\right)}\,, \\
H_t^{-1} \partial_t\, H_t  \,\, = & \,\,\, \frac{\omega_t^2\wedge A_t^\kappa}{\omega_t^3}\,,
\end{aligned}
\end{equation}
where $Rm^\tau$ and $A^\kappa$ are, respectively, the curvature tensors of Gauduchon connections $\nabla^\tau$ on $(X,\omega_t)$ and $\nabla^\kappa$ on $(E,H_t)$, and $\alpha'\in\RR$ is the so-called {\em slope parameter}.\medskip

Phong, Picard and Zhang proved that, if the connections $\nabla^\tau$ and $\nabla^\kappa$ in \eqref{AF-G} are both Chern, then the flow preserves the {\em conformally balanced} condition $d(\Vert\Psi\Vert_{\omega}\,\omega^2)=0$ and, under an extra assumption on the initial metric $\omega_0$, it is well-posed \cite{PPZ18}. If furthermore $\omega_0$ is conformally balanced and the flow is defined for every $t\in[0,\infty)$, then its limit points $(\omega_\infty,H_\infty)$ are automatically solutions to the Hull-Strominger system \cite{PPZ18}.

More generally, any stationary solution to the Anomaly flow \eqref{AF-G}, satisfying the conformally balanced condition and for which the curvature form $A^\kappa$ is of type $(1,1)$, is a solution to the {\em Hull-Strominger system} \cite{Hull86,Strom}:
\begin{equation}\label{HS}
\begin{aligned}
&\omega^2\wedge A^\kappa=0\,,\quad (A^\kappa)^{2,0}=(A^\kappa)^{0,2}=0\,,\\
&i\, \partial\overline\partial \omega = \frac{\alpha'}{4}\left( {\rm Tr}(Rm^\tau\wedge Rm^\tau) - {\rm Tr}(A^\kappa\wedge A^\kappa)\right)\,,\\
&d(\Vert\Psi\Vert_{\omega}\,\omega^2)=0\,.
\end{aligned}
\end{equation}
Here, the first two equations represent the {\em Hermitian-Yang-Mills equation} for the connection $\nabla^\kappa$; the third equation follows by the Green-Schwarz cancellation mechanism in string theory and it is known as {\em anomaly cancellation}; while, the last equation was originally formulated as
$$
d^\ast \omega=i(\bar\partial-\partial)\ln\Vert\Psi\Vert_\omega\,,
$$
where $d^\ast$ is the co-differential, and the above equivalent expression is due to Li and Yau \cite{LY05}. 

We mention that the Hull-Strominger system arises from the symmetric compactification of the 10-dimensional heterotic string theory and it has been extensively studied both from physicists and mathematicians (see e.g.  \cite{AG12,AG12_2,Fei16,FPZ17,FIUV14,FGV19,FY07, FY08,notasM,OUV,UV14}).\medskip

The Anomaly flow turned out to be a powerful tool in the study of the Hull-Strominger system. In particular, it was used to give an alternative proof of the outstanding results obtained by Fu and Yau in \cite{FY07,FY08} about the existence of solutions to \eqref{HS}. More precisely, in \cite{PPZ18_3} Phong, Picard and Zhang studied the flow on a torus fibration over a K3 surface, showing that if $\omega_0$ is conformally balanced and satisfies some extra assumptions, then the flow has a long-time solution which always converges to a solution of the Hull-Strominger system, once the connections $\nabla^\tau$ and $\nabla^\kappa$ are both Chern.

In a attempt to better understand the `general' behaviour of the Anomaly flow, a simplified version of the flow with `flat' bundle was proposed in \cite{PPZ18_2}, namely 
\begin{equation}\label{AF-triv}
\partial_t(\Vert\Psi\Vert_{\omega_t}\,\omega_t^2)\,\, = \,\,  i\partial\overline\partial \omega_t - \frac{\alpha'}{4}\, {\rm Tr}(Rm_t^\tau\wedge Rm_t^\tau)\,.
\end{equation}
Then, following the approach proposed by Fei and Yau in \cite{FY15} to solve the Hull-Strominger system on complex unimodular Lie groups, Phong, Picard and Zhang investigated this new flow on such Lie groups \cite{PPZ19}.\medskip

In the present paper we study the behaviour of the Anomaly flows \eqref{AF-G} and \eqref{AF-triv} on a class of nilmanifolds. We will assume the trace ${\rm Tr}(A_t^\kappa\wedge A_t^\kappa)$ to be of a special type (see Assumption \ref{condition}) and all the involved structures are to be intended {\em invariant}.

Our first result characterizes the solutions to the Anomaly flows under our hypotheses.
\begin{thmint}\label{thm_A} Let $M=\Gamma\,\backslash\, G$ be a $2$-step nilmanifold of dimension $6$ with first Betti number $b_1\geq4$. Let $J$ be a non-parallelizable complex structure on $M$. Then, there always exists a preferable $(1,0)$-coframe $\{\zeta^i\}$ on $X=(M,J)$ such that the family of Hermitian metrics $\omega_t$ solving \eqref{AF-G} or \eqref{AF-triv} is given by
$$
\omega_t =\frac i2\,r(t)^2 \left( \zeta^{1\bar{1}} + a\, \zeta^{2\bar{2}}+b\, \zeta^{1\bar{2}}+\bar b\, \zeta^{2\bar{1}} \right)+ \frac i2\, c\, \zeta^{3\bar{3}}\,,
$$
for some $a,c\in\RR$ and $b\in\C$ depending on $\omega_0$. Furthermore, if $\omega_t$ is a solution to the Anomaly flow~\eqref{AF-triv}, then $r(t)^2$ solves the ODE
\begin{equation}\label{MP-intro}
\frac d{dt}\, r(t)^2= K_1+\frac{ K_2}{r(t)^4}\,, 
\end{equation}
with $K_{1}, K_2\in\RR$ constants depending on $K_1=K_1(\omega_0)$ and $K_2=K_2(\omega_0,\alpha',\tau)$.
\end{thmint}

As a direct consequence, we have that 
the qualitative behaviour of the Anomaly flow \eqref{AF-triv} only depends on signs of the constants $K_1$ and $K_2$. Remarkably, this implies that the Anomaly flow \eqref{AF-triv} admits both immortal and ancient invariant solutions on the same nilmanifolds, and, as far as we know, this provides the second example of a metric flow with such a property (the first one is the pluriclosed flow \cite{AL}). 
We also show that the constant $K_1$ is closely related to a conformal invariant of the metric $\omega_0$ introduced and studied in~\cite{FWW}.\smallskip 
%(see Section~\ref{signosdeKs}).\smallskip

Our second result is about the convergence of the Anomaly flow \eqref{AF-triv} in Gromov-Hausdorff topology. In particular, under the same assumption of Theorem \ref{thm_A}, we have

\begin{thmint}\label{thm_conv} Let $\omega_t$ be an immortal solution to the Anomaly flow \eqref{AF-triv}. Then, $(X, (1+t)^{-1}\omega_t)$ converges either to a point or to a real torus $\mathbb T^4$ in the Gromov-Hausdorff topology as $t\to +\infty$, depending on the initial metric $\omega_0$ and the signs of $K_1$ and $K_2$ in \eqref{MP-intro}.
\end{thmint}

It is worth noting that, Theorem \ref{thm_A} holds for any initial invariant Hermitian metric $\omega_0$ on $X$. Nonetheless, in view of the Hull-Strominger system, it would be desirable for the Anomaly flows to preserve the locally conformally balanced condition. Under the assumptions of Theorem \ref{thm_A}, our third result states as follows

\begin{thmint}\label{thm_B} The Anomaly flows \eqref{AF-G} and \eqref{AF-triv} preserve the balanced condition.
\end{thmint}

%We will also show that the {\em locally conformally K\"ahler} condition is preserved along the flows \eqref{AF-G} and \eqref{AF-triv}.\medskip

The last part of this paper is devoted to the study of the Anomaly flow \eqref{AF-G} in some interesting non-flat cases. In particular, we consider a class of Lie groups $G$ arising from Theorem \ref{thm_A}, and we equip them with holomorphic tangent bundles, that is $E=T^{1,0}G$. 

\begin{thmint}\label{thm_C} If the initial metrics $(\omega_0,H_0)$ are both diagonal, then $\omega_t$ and $H_t$ hold diagonal along the Anomaly flow \eqref{AF-G}.
\end{thmint}

As a relevant application of this theorem, we obtain solutions to the field equations of the heterotic string on the nilmanifold corresponding to the nilpotent Lie group $N_3$. 
In \cite{Iv} Ivanov proved that, in order to solve such field equations, the solution to the Hull-Strominger system \eqref{HS} must satisfy the extra condition
that $Rm$ is the curvature of an $\rm SU(3)$-instanton with
respect to $\omega$ (see also \cite{FIUV}). In our setting, this means that the curvature of the
Gauduchon connection $\nabla^\tau$ has to satisfy
\begin{equation}\label{extra-HS}
\begin{aligned}
\omega^2\wedge Rm^\tau=0\,,\qquad (Rm^\tau)^{2,0}=(Rm^\tau)^{0,2}=0\,.
\end{aligned}
\end{equation}
%%%Hence, the solutions to the system \eqref{HS} satisfying in addition %%%\eqref{extra-HS}, provide solutions to the heterotic equations of motion. 
Following \cite{PPZ18_4}, we will refer to the whole system given by \eqref{HS} and  \eqref{extra-HS} as the {\em Hull-Strominger-Ivanov system}. 
Explicit solutions to this system were found in \cite{FIUV} on nilmanifolds, and more recently in~\cite{OUV} on solvmanifolds and on the quotient of SL(2,$\C$). 

As an application of Theorem~\ref{thm_C}, given the nilpotent Lie group $N_3$, we prove that if $\nabla_t^\kappa$ is the Strominger-Bismut connection of $H_t$ and the initial metric $\omega_0$ is balanced, then the flow reduces to an ODE of the form of \eqref{MP-intro}. This allows us to prove that the Anomaly flow \eqref{AF-G} always converges to a solution of the Hull-Strominger-Ivanov system when  $\nabla_t^\tau$ is the Strominger-Bismut connection of the metric $\omega_t$ (see Theorem \ref{h3-Bismut-instanton}).

\medskip

It is worth noting that a generalization of the Anomaly flow to Hermitian manifolds of any dimension has been proposed in \cite{PPZ19_2}; while, in \cite{FP19_2} Fei and Phong proved that this generalized Anomaly flow is related to a flow in the Hermitian curvature flows family, that is, a family of parabolic flows introduced by Streets and Tian in \cite{ST}. We mention that the Hermitian curvature flow related to the Anomaly flow has been studied on 2-step nilpotent complex Lie groups by the first named author in \cite{P20}, who proved long-time existence and convergence results. We also refer to \cite{AL, EFV, LPV, PP, P19, PV18, Ust} for some recent results on Hermitian curvature flows in different homogeneous settings.\bigskip

The paper is organized as follows. Section \ref{prelim} is devoted to basic computations on our class of nilpotent Lie groups. In particular, we show that under our assumptions we can always find a preferable real coframe on such Lie groups, namely an {\em adapted basis}. Then, by using this basis, we explicitly compute ${\rm Tr}(Rm^\tau\wedge Rm^\tau)$. In Section \ref{sec-flow} we begin the study of the Anomaly flows and we prove Theorem \ref{thm_A} and Theorem \ref{thm_B}. In Section 4 we focus on the Anomaly flow \eqref{AF-triv}, showing that we can always reduce the flow to the ODE \eqref{MP-intro}. We also study the qualitative behaviour of the Anomaly flow \eqref{AF-triv} depending on the signs of $K_1$, $K_2$. These results will in turn imply Theorem~\ref{thm_conv}. In Section~\ref{sec-flow-example} we study the Anomaly flow \eqref{AF-G} on a special class of nilpotent Lie groups and we prove Theorem~\ref{thm_C}. We also investigate the behaviour of the flow on an explicit example. Finally, 
Appendix~A and Appendix~B
%Appendix \ref{apendiceA} and Appendix \ref{apendiceB} 
contain some technical computations which we used in the paper.

\section{Preliminaries}\label{prelim}

\noindent
In this section, we consider a nilpotent Lie group $G$ of (real) dimension 6 endowed with a left-invariant Hermitian structure $(J,\omega)$. In particular, we will find a coframe which adapts to the Hermitian structure, allowing us to explicitly compute the trace of the curvature for any Hermitian connection in the Gauduchon family. \smallskip

Let $G$ be a 6-dimensional Lie group equipped with a left-invariant complex structure $J$ and a left-invariant Hermitian metric $\omega$. Let $\{ Z_1,Z_2,Z_3\}$ be a left-invariant (1,0)-frame of $(G,J)$, and $\{\zeta^1,\zeta^2,\zeta^3\}$ its dual frame. Then, we can always write
\begin{equation}\label{2forma}
2\, \omega=i (r^2\,\zeta^{1\bar{1}} + s^2\,\zeta^{2\bar{2}} +k^2\,\zeta^{3\bar{3}})+u\,\zeta^{1\bar{2}}-\bar{u}\,\zeta^{2\bar{1}} +v\,\zeta^{2\bar{3}}-\bar{v}\,\zeta^{3\bar{2}}+z\,\zeta^{1\bar{3}}-\bar{z}\,\zeta^{3\bar{1}}\,,
\end{equation}
where $r,s,k\in\RR^\ast$ and $u,v,z\in\C$ satisfy
\begin{equation}\label{F-non-deg-1}
r^2s^2>|u|^2,\quad s^2k^2>|v|^2,\quad r^2k^2>|z|^2\,,
\end{equation}
and
\begin{equation}\label{F-non-deg-2}
8i\,\det \omega=r^2s^2k^2 + 2\,\Re(i\bar u\bar v z) - k^2|u|^2 - r^2|v|^2 - s^2|z|^2>0
\end{equation}
by the positive definiteness of the metric. Here
$$
\det \omega=\frac 18\det \left(\!\!\! \begin{array}{ccc}
i\,r^2 & u & z \\
-\overline{u} & i\,s^2 & v \\
-\overline{z} & -\overline{v}& i\, k^2
\end{array} \!\right)\!.
$$

In the following, we focus on Lie groups $G$ which are 2-step nilpotent and such that  $(G,J)$ is not a \emph{complex Lie group}, that is, the left-invariant complex structure $J$ is not complex parallelizable. Notice that, the latter condition excludes just two cases (the complex torus and the Iwasawa manifold), both of which have already been studied in \cite{PPZ19} (see also \cite{P20}). Additionally, we will suppose that the dimension of the first Chevalley-Eilenberg cohomology group $H^1(\frg)$ of the Lie algebra $\frg$ of $G$ is at least 4. We denote such dimension by $b_1(\frg)$, since by the Nomizu theorem it coincides with the first Betti number of the nilmanifold $\nilm$ obtained as the quotient of $G$ by a co-compact lattice $\Gamma$.

Under these assumptions we are able to study a large family of Hermitian metrics in a unified setting. 
Indeed, by means of \cite[Proposition 2]{U} (see also \cite[Proposition 2.4]{COUV}),
if $J$ is not complex parallelizable and $b_1(\frg)=\dim H^1(\frg)\geq 4$, then there exists a left-invariant (1,0)-coframe $\{\zeta^1,\zeta^2,\zeta^3\}$ on $(G,J)$ satisfying
\begin{equation}\label{J-nilp}
\begin{cases}
d \zeta^1=d\zeta^2=0,\cr d\zeta^3=\rho\, \zeta^{12} +
\zeta^{1\bar{1}} + \lambda\,\zeta^{1\bar{2}} + (x+i\,y)\,\zeta^{2\bar{2}}\,,
\end{cases}
\end{equation}
where $x,y,\lambda\in \mathbb{R}$ with $\lambda\geq 0$, and $\rho\in \{0,1\}$. 
(See Table~\ref{table-sign} in Section~\ref{signosdeKs} for a description of all the real Lie groups supporting such Hermitian structures.)

\subsection{Adapted bases}\hfill\medskip

\noindent 
Our first result shows that we can always find a preferable (real) left-invariant coframe $\{e^1,\ldots,e^6\}$ on $G$ associated to any left-invariant Hermitian structure $(J,\omega)$. In the following, we refer to such a coframe as {\em adapted basis}.

\begin{proposition}\label{adapted-ecus}
Let $G$ be a $2$-step nilpotent Lie group of dimension $6$ with $b_1(\frg)\geq4$, $\frg$ being the Lie algebra of $G$. Let $J$ be a left-invariant non-parallelizable complex structure on $G$ and $\omega$ a left-invariant $J$-Hermitian metric. 
Suppose that $J$ and $\omega$ are defined by \eqref{J-nilp} and \eqref{2forma}, respectively. 
%%%in terms of the (1,0)-coframe $\{\omega^1,\omega^2,\omega^3\}$ on $(G,J)$. 
Then, there exists a (real) left-invariant coframe $\{e^1,\ldots,e^6\}$ on $G$, such that:
\begin{enumerate}\setlength\itemsep{0.5em}
\item[{\rm (a)}]
The complex structure $J$ and the metric $\omega$ satisfy
\begin{equation}\label{adapted-basis}
Je^1=-e^2,\ Je^3=-e^4,\ Je^5=-e^6,\quad\quad \omega=e^{12}+e^{34}+e^{56}.
\end{equation}
\item[{\rm (b)}]
The coframe satisfies the following structure equations
\begin{equation}\label{J-nilp-real-basis}
\left\lbrace\,\,
\begin{aligned}
d e^1 =&\ d e^2=d e^3=d e^4=0\,,\\
d e^5 =&\ \tfrac{k_e}{\Delta_e} \left( \rho+\lambda \right)\, e^{13} - \tfrac{k_e}{\Delta_e} \left( \rho-\lambda \right)\, e^{24} + \tfrac{2\,k_e}{\Delta_e^2} \left( r_e^2\, y - \lambda\, u_{e1} \right) e^{34}\,,\\
d e^6 =&- \tfrac{2k_e}{r_e^2}\, e^{12} + \tfrac{2 k_e u_{e1}}{r_e^2 \Delta_e}\, e^{13} +\tfrac{k_e}{r_e^2 \Delta_e} \left( r_e^2(\rho-\lambda)+2 u_{e2} \right)\, e^{14}\,\\
&  +\tfrac{k_e}{r_e^2 \Delta_e} \left( r_e^2(\rho+\lambda)-2 u_{e2} \right)\, e^{23} +\tfrac{2 k_e u_{e1}}{r_e^2 \Delta_e}\, e^{24}\,,\\
&  - \tfrac{2\, k_e }{r_e^2 \Delta_e^2} \left( r_e^4\, x - \lambda\, r_e^2\, u_{e2} +u_{e1}^2+u_{e2}^2 \right) e^{34}.
\end{aligned}
\right.
\end{equation}
Here, $x,y,\lambda\in \mathbb{R}$ with $\lambda\geq 0$, and $\rho\in\{ 0,1\}$ are the coefficients in \eqref{J-nilp} which define the complex structure $J$, whereas the coefficients $r_e,s_e,k_e,u_{e1},u_{e2} \in \mathbb{R}$, which depends on the coefficients of $\omega$, are given by 
\begin{equation}\label{new-coeffs-1}
r_{e}^2=r^2-\frac{|z|^2}{k^2}\,,\quad s_{e}^2=s^2-\frac{|v|^2}{k^2}\,,\quad k_{e}^2=k^2, \quad  u_{e1}+i\,u_{e2} :=u_{e}=u-\frac{i\bar vz}{k^2}\,.
\end{equation} 
%%%Note that if the metric coefficients $v=z=0$ in \eqref{2forma}, then $r_e=r$,  %%%$s_e=s$, $k_e=k$, $u_e=u$.
%\footnote{From \eqref{new-coeffs-1} we have that if the metric coefficients $v=z=0$ in \eqref{2forma}, then in the adapted basis the structure constants in \eqref{J-nilp-real-basis} are $r_e=r$,  $s_e=s$, $k_e=k$, $u_e=u$.}
\noindent The term $\Delta_e$ in the equations \eqref{J-nilp-real-basis} stands for $\Delta_e:=\sqrt{r_e^2s_e^2-|u_{e}|^2}=\sqrt{\frac{8i\,\det \omega}{k^2}}$.
\item[{\rm (c)}]
The $4$-form $e^{1234}$ is a positive multiple of the $(2,2)$-form $\zeta^{12\bar{1}\bar{2}}$, concretely
\begin{equation}\label{relacion}
e^{1234}=\frac{2i \, \det \omega}{k^2}\, \zeta^{12\bar{1}\bar{2}}.
\end{equation}
%
%%%\item[{\rm (d)}]
%%%If $v=z=0$ in \eqref{2forma}, then $r_e=r$,  $s_e=s$, $k_e=k$ and %%%$u_e=u$.
\end{enumerate}
\end{proposition}

\begin{proof}
Starting from a left-invariant $(1,0)$-coframe $\{\zeta^1,\,\zeta^2,\,\zeta^3\}$ satisfying 
\eqref{J-nilp} and a generic $J$-Hermitian metric $\omega$ in the form of \eqref{2forma}, we first consider the left-invariant (1,0)-coframe
\begin{equation}\label{J-nilp-sigma-anterior}
\sigma^1:=\zeta^1\,,\quad\sigma^2:=\zeta^2\,,\quad\sigma^3:=\zeta^3-\frac{iv}{k^2}\,\zeta^2-\frac{iz}{k^2}\,\zeta^1\,.
\end{equation}
This map defines an automorphism of the complex structure $J$ which preserves the complex structure equations \eqref{J-nilp}, i.e. the $(1,0)$-coframe $\{\sigma^1,\,\sigma^2,\,\sigma^3\}$ still satisfies 
%%%\begin{equation}\label{J-nilp-sigma}
%%%\begin{cases}
%%%d \sigma^1=d\sigma^2=0,\cr d\sigma^3=\rho\, \sigma^{12} +
%%%\sigma^{1\bar{1}} + \lambda\,\sigma^{1\bar{2}} + D\,\sigma^{2\bar{2}},%%%\cr
%%%\end{cases}
%%%\end{equation}
\begin{equation}\label{J-nilp-sigma}
d \sigma^1=d\sigma^2=0,\quad\quad 
d\sigma^3=\rho\, \sigma^{12} +
\sigma^{1\bar{1}} + \lambda\,\sigma^{1\bar{2}} + (x+i\,y)\,\sigma^{2\bar{2}}.
\end{equation}
With respect to this coframe, the Hermitian metric $\omega$ can be written as
\begin{equation}\label{J-nilp-sigma-posterior}
2\, \omega=i\,(r_{\sigma}^2\,\sigma^{1\bar 1} + s_{\sigma}^2\,\sigma^{2\bar2} +k_{\sigma}^2\,\sigma^{3\bar3}) + u_{\sigma}\,\sigma^{1\bar2} - \overline{u}_{\sigma}\,\sigma^{2\bar1},
\end{equation}
where the metric coefficients are given by
\begin{equation}\label{new-coeffs}
r_{\sigma}^2=r^2-\frac{|z|^2}{k^2}\,,\quad s_{\sigma}^2=s^2-\frac{|v|^2}{k^2}\,,\quad k_{\sigma}^2=k^2, \quad u_{\sigma}=u-\frac{i\bar vz}{k^2}\,.
\end{equation} 
Notice that from \eqref{F-non-deg-1} we have 
$r_{\sigma}^2, s_{\sigma}^2, k_{\sigma}^2 >0$ and  $r_{\sigma}^2s_{\sigma}^2>|u_{\sigma}|^2$. 

Let us now consider the left-invariant $(1,0)$-coframe
\begin{equation}\label{J-taus}
\tau^1:=r_{\sigma}\, \sigma^1+\frac{i\,\overline{u}_{\sigma}}{r_{\sigma}}\, \sigma^2,\quad \tau^2:=\frac{\Delta_{\sigma}}{r_{\sigma}}\, \sigma^2,\quad \tau^3:=k_{\sigma}\, \sigma^3\,,
\end{equation}
where $\Delta_{\sigma}:=\sqrt{r_{\sigma}^2 s_{\sigma}^2-|u_{\sigma}|^2}$.
%%%with $\Delta_{\sigma}:=\sqrt{r_{\sigma}^2 s_{\sigma}^2-|u_{\sigma}|^2}$.
Then, a direct calculation yields that $\omega$ can be written as
$$
\omega=\frac{i}{2}\,\tau^{1\bar 1} + \frac{i}{2}\,\tau^{2\bar 2} +\frac{i}{2}\,\tau^{3\bar 3}
$$
and, by using \eqref{J-nilp-sigma}, the complex structure equations become
\begin{equation}\label{J-nilp-tau}
\left\{
\begin{array}{lcl}
d \tau^1 \!\!\!&=&\!\!\! d\tau^2=0,\\[6pt] 
d\tau^3 \!\!\!&=&\!\!\!  \rho\,\frac{k_{\sigma}}{\Delta_{\sigma}}\, \tau^{12} + \frac{k_{\sigma}}{r_{\sigma}^2}\, \tau^{1\bar{1}} + \frac{k_{\sigma}}{r_{\sigma}^2\Delta_{\sigma}} \left( i u_{\sigma}+\lambda r_{\sigma}^2 \right)\,\tau^{1\bar{2}} - \frac{i k_{\sigma} \overline{u}_{\sigma}}{r_{\sigma}^2\Delta_{\sigma}}\,\tau^{2\bar{1}}\\[7pt]
\!\!\!&&\!\!\!  + \frac{k_{\sigma}}{r_{\sigma}^2\Delta_{\sigma}^2} 
\left( |u_{\sigma}|^2 - i r_{\sigma}^2 \overline{u}_{\sigma} \lambda + r_{\sigma}^4 x+i\, r_{\sigma}^4 y \right)\tau^{2\bar{2}}.
\end{array}
\right.
\end{equation}
Finally, let us consider the real left-invariant coframe $\{ e^1,\ldots,e^6 \}$ on $G$ given by
\begin{equation}\label{es-taus}
e^1+i\,e^2:=\tau^1,\quad\  e^3+i\,e^4:=\tau^2,\quad\  e^5+i\,e^6:=\tau^3.
\end{equation}
Then, with respect to this real coframe, \eqref{adapted-basis} holds and hence (a) is proved.\smallskip

Now, let us set $u_{\sigma 1}+i\, u_{\sigma 2}:=u_{\sigma}$. By means of \eqref{J-nilp-tau}, a direct computation yields that the structure equations in terms of the real coframe $\{ e^1,\ldots,e^6 \}$ are given by
$$
\left\{
\begin{array}{rcl}
d e^1 \!\!\!&=&\!\!\!  d e^2=d e^3=d e^4=0,\\[6pt]
\frac{\Delta_{\sigma}}{k_{\sigma}}\, d e^5 \!\!\!&=&\!\!\! \left( \rho+\lambda \right) e^{13} 
- \left( \rho-\lambda \right) e^{24} 
+ \frac{2}{\Delta_{\sigma}} 
\left( r_{\sigma}^2\, y - \lambda\, u_{\sigma 1} \right) e^{34},\\[7pt]
\frac{r_{\sigma}^2 \Delta_{\sigma}}{k_{\sigma}}\, d e^6 \!\!\!&=&\!\!\! -2 \Delta_{\sigma}\, e^{12} 
+ 2 u_{\sigma 1}\, e^{13} 
+ \left( r_{\sigma}^2(\rho-\lambda)+2 u_{\sigma 2} \right) e^{14} 
+ \left( r_{\sigma}^2(\rho+\lambda)-2 u_{\sigma 2} \right) e^{23}\\[7pt]
\!\!\!&&\!\!\!
+2 u_{\sigma 1}\, e^{24}  
- \frac{2}{\Delta_{\sigma}} 
\left( r_{\sigma}^4\, x - \lambda r_{\sigma}^2 u_{\sigma 2} +u_{\sigma 1}^2+u_{\sigma 2}^2 \right) e^{34}\,.
\end{array}
\right.
$$
Therefore, setting $r_e:=r_{\sigma}$, $s_e:=s_{\sigma}$, $k_e:=k_{\sigma}$, $u_{e}:=u_{\sigma}$ (thus, $u_{e1}:=u_{\sigma 1}$ and $u_{e2}:=u_{\sigma 2}$) 
and $\Delta_{e}=\Delta_{\sigma}$, we get \eqref{J-nilp-real-basis} and (b) follows. Notice that \eqref{new-coeffs-1} is precisely 
\eqref{new-coeffs} in the new notation. Moreover, from \eqref{new-coeffs-1} we get
$$
\Delta_{e}^2=r_{e}^2s_{e}^2-|u_{e}|^2=\frac{1}{k^2} \left( r^2s^2k^2 + 2\,\Re(i\bar u\bar v z)-k^2|u|^2 - r^2|v|^2 - s^2|z|^2 \right)=\frac{8i\,\det \omega}{k^2},
$$
due to \eqref{F-non-deg-2}. 

Finally, in order to prove (c), it is enough to note that 
$
4\,e^{1234}=\tau^{12\bar{1}\bar{2}}=\Delta_{e}^2\, \zeta^{12\bar{1}\bar{2}}$.
\end{proof}

\begin{remark}
{\rm
In the balanced case, adapted bases were found in \cite[Theorem 2.11]{UV15} by considering a partition of the space of metrics into the subsets ``$u=0$'' and ``$u\not=0$'' for the given left-invariant metrics \eqref{2forma}. However, the study 
of the Anomaly flow requires to consider a global setting involving the whole space of left-invariant Hermitian metrics $\omega$ on $(G,J)$, as it has been obtained in our 
Proposition~\ref{adapted-ecus}.
}
\end{remark}

\subsection{Trace of the curvature}\label{sub-trace}\hfill\medskip

\noindent 
In the following, we explicitly compute the trace of the curvature of a Hermitian connection in the Gauduchon family $\{\nabla^{\tau}\}_{\tau\in \mathbb{R}}$ for our class of nilpotent Lie groups. We will adopt the convention used in \cite{FIUV} (see also \cite{OUV}). \medskip

Given a smooth manifold $M$ equipped with a Hermitian structure $(J,\omega)$, 
a connection $\nabla$ on the tangent bundle $TM$ is said to be a 
{\em Hermitian connection} if 
$\nabla J=0$ and $\nabla \omega=0$.
Gauduchon introduced in \cite{Gaud} a 1-parameter family $\{\nabla^{\tau}\}_{\tau\in \mathbb{R}}$ of {\em canonical} Hermitian connections, which can be defined via
\begin{equation}\label{family_G}
\omega(J(\nabla^{\tau}_XY),Z) = \omega(J(\nabla^{LC}_X Y), Z)  + \frac{1-\tau}{4}\,T(X,Y,Z) + \frac{1+\tau}{4}\,C(X,Y,Z) ,
\end{equation}
where $\nabla^{LC}$ is the Levi-Civita connection of the Riemannian manifold $(M,\omega)$, and $T$ and $C$ are given by
$$
T(\cdot, \cdot, \cdot):=-d\omega(J\cdot,J\cdot,J\cdot) \qquad\text{and}\qquad
C(\cdot, \cdot, \cdot):=d\omega(J\cdot, \cdot, \cdot).
$$ 
These connections are distinguished by the properties of their torsion tensors and, in view of \eqref{family_G}, the Chern connection $\nabla^c$ and the Strominger-Bismut connection $\nabla^+$ can be recovered by taking $\tau = 1$ and $\tau= -1$, respectively.\smallskip

Now, let us consider a 6-dimensional Lie group $G$ equipped with a left-invariant Hermitian structure $(J,\omega)$. Let also $\{e^1,\ldots,e^6\}$ be 
%%%an invariant real coframe on $G$ adapted 
an adapted basis to the Hermitian structure on $G$, i.e. $J$ and $\omega$ are expressed by \eqref{adapted-basis}.
The {\em connection $1$-forms} $\sigma^i_j$ associated to any linear connection $\nabla$ are
$$
\sigma^i_j(e_k) := \omega(J(\nabla_{e_k}e_j), e_i),
$$
that is, $\nabla_X e_j = \sigma^1_j(X)\, e_1 +\cdots+ \sigma^6_j(X)\, e_6$; while, the {\em curvature $2$-forms} $\Omega^i_j$ of $\nabla$ are given by 
\begin{equation}\label{curvature}
\Omega^i_j := d \sigma^i_j + \sum_{1\leq k \leq 6}\sigma^i_k\wedge \sigma^k_j.
\end{equation}
Then, the trace of the 4-form $\Omega\wedge\Omega$ can be defined via
\begin{equation}\label{traceRmRm}
{\rm Tr}(\Omega\wedge\Omega)= \sum_{1\leq i<j\leq 6} \Omega^i_j\wedge\Omega^i_j.
\end{equation}

Remarkably, the connection 1-forms $(\sigma^\tau)^i_j$ associated to a canonical connection $\nabla^\tau$ in Gauduchon family can be explicitly obtained as follows. Let us denote by $c_{ij}^k$ the structure constants of $G$ with respect to $\{e^1,\ldots,e^6\}$, i.e.
 $$
d\,e^k = \sum_{1\leq i<j \leq 6} c_{ij}^k \, e^{i j},\quad\quad k=1,\ldots,6.
$$
Since $d e^k(e_i,e_j)= -e^k([e_i,e_j])$, the connection 1-forms $(\sigma^{LC})^i_j$ of the Levi-Civita connection satisfy
$$
(\sigma^{LC})^i_j(e_k) = -\frac12 \left( \omega(Je_i,[e_j,e_k]) - \omega(Je_k,[e_i,e_j]) +\omega(Je_j,[e_k,e_i]) \right)=\frac12(c^i_{jk}-c^k_{ij}+c^j_{ki}),
$$
and hence, by means of \eqref{family_G}, the connection 1-forms $(\sigma^{\tau})^i_j$ of  ${\nabla^{\tau}}$ are given by
$$\begin{aligned}
 (\sigma^{\tau})^i_j(e_k)=&(\sigma^{LC})^i_j(e_k) - \frac{1-\tau}{4}\, T(e_i,e_j,e_k) - \frac{1+\tau}{4}\, C(e_i, e_j, e_k) \\ 
 =&\frac12(c^i_{jk}-c^k_{ij}+c^j_{ki}) -\frac{1-\tau}{4}\, T(e_i, e_j, e_k) - \frac{1+\tau}{4}\, C(e_i,e_j,e_k)\\ 
 =&\frac12(c^i_{jk}-c^k_{ij}+c^j_{ki}) +\frac{1-\tau}{4}\,d\omega(J e_i, J e_j, J e_k) - \frac{1+\tau}{4}\, d\omega(J e_i,e_j,e_k).
\end{aligned}$$

Henceforth, the curvature of the Gauduchon connection $\nabla^{\tau}$ will be  denoted indistinctly by $Rm^{\tau}$ or $\Omega^{\tau}$. 
We are now in a position to compute the trace of $\Omega^{\tau}\!\wedge\Omega^{\tau}$ for our class of nilpotent Lie groups. In our next proposition, we show that the trace is of a special type. This will allow us to substantially simplify the Anomaly flow equations.

\begin{proposition}\label{traza-tau}
Let $G$ be a $2$-step nilpotent Lie group of dimension $6$ with $b_1(\frg)\geq4$, $\frg$ being the Lie algebra of $G$. Let $J$ be a left-invariant non-parallelizable complex structure on $G$ and $\omega$ a left-invariant $J$-Hermitian metric. 
Suppose that $J$ and $\omega$ are defined, respectively, by \eqref{J-nilp} and \eqref{2forma} in terms of a left-invariant $(1,0)$-coframe $\{ \zeta^l\}_{l=1}^3$. Then, for any Gauduchon connection $\nabla^{\tau}$, the trace of its curvature satisfies 
$$
{\rm Tr}(\Omega^{\tau}\!\wedge\Omega^{\tau}) = C \, \zeta^{12\bar{1}\bar{2}},
$$
where $C=C(\rho,\lambda,x,y;r,s,k,u,v,z;\tau)$ is a constant depending both on the Hermitian structure and the connection. 
More precisely, we have
\begin{equation*}\begin{aligned}
{\rm Tr}(\Omega^{\tau}\!\wedge\,&\Omega^{\tau}) = -\displaystyle{\frac{(\tau-1)\,k_e^4}{2(r_e^2 s_e^2-|u_e|^2)^2}} \, \Big\{\\
&\Big[ (\rho-\lambda^2+5x)  (s_e^4 \!-\! 2 \lambda s_e^2 u_{e2} + 2 x |u_e|^2) - 3 \lambda^2 x (u_{e1}^2\!-\!u_{e2}^2) -6 \lambda u_{e1} y (s_e^2\!-\!\lambda u_{e2}) + 6 y^2 |u_e|^2 \\
&+\tau (\rho+\lambda^2-2x) (s_e^4 - 2 \lambda s_e^2 u_{e2} + 2 x |u_e|^2) \\
&+\tau^2 \Big(  (-2\rho+x) (s_e^4 \!-\!2 \lambda s_e^2 u_{e2} +2 x |u_e|^2) -\lambda^2 x (u_{e1}^2\!-\!u_{e2}^2) -2 \lambda u_{e1} y (s_e^2\!-\!\lambda u_{e2}) + 2 y^2 |u_e|^2  \Big) \Big]\\
&+r_e^2\, \lambda \Big[ (\rho-\lambda^2+2x) (\lambda s_e^2 - 2 u_{e2} x - 2 u_{e1} y) - 6 u_{e2}(x^2 + y^2)\\
&\qquad\quad+\tau  (\rho+\lambda^2-2x) (\lambda s_e^2 - 2 u_{e2} x - 2 u_{e1} y) \\
&\qquad\quad+\tau^2 \Big( \!\!-2\rho (\lambda s_e^2 -2 u_{e2} x -2 u_{e1} y) -2  u_{e2} (x^2 + y^2) \Big) \Big]\\
&+r_e^4 (x^2 +y^2) \Big[ (\rho - \lambda^2+5x) +\tau \, (\rho + \lambda^2-2x)  +\tau^2 \, (-2\rho+x) \Big]  \, \Big\} \, \zeta^{12\bar{1}\bar{2}}\,,
\end{aligned}\end{equation*}
where $r_e,s_e,k_e$ and $u_{e}= u_{e1}+i\,u_{e2}$ are given in Proposition \ref{adapted-ecus}.
\end{proposition}

\begin{proof}
By means of Proposition \ref{adapted-ecus}, there always exists an adapted basis $\{e^1,\ldots,e^6\}$ on the Lie group $G$ for the left-invariant Hermitian structure $(J,\omega)$. Therefore, let $(\sigma^{\tau})^i_j$ be the connection 1-forms of the Gauduchon connection $\nabla^{\tau}$ in this basis. Since $\nabla^{\tau}$ is Hermitian, then the forms $(\sigma^\tau)^i_j$ satisfy the condition $(\sigma^{\tau})^i_j = - (\sigma^{\tau})^j_i$. Moreover, a direct  computation yields the following connection 1-forms:
$$
\begin{aligned}
(\sigma^{\tau})^1_2=& -\tfrac{k_e}{r_e^2}{\scriptstyle(\tau-1)}\, e^6\,,\\[3pt]
(\sigma^{\tau})^1_3=&\ \tfrac{\lambda\, k_e}{2\sqrt{r_e^2s_e^2-|u_e|^2}}{\scriptstyle(\tau-1)}\,  e^5 + \tfrac{k_e\, u_{e1}}{r_e^2\sqrt{r_e^2s_e^2-|u_e|^2}}{\scriptstyle(\tau-1)} \, e^6\,,\\[3pt]
(\sigma^{\tau})^1_4=& -\tfrac{k_e (\lambda\, r_e^2-2\,u_{e2})}{2r_e^2\sqrt{r_e^2s_e^2-|u_e|^2}}{\scriptstyle(\tau-1)} \,e^6\,,\\[3pt]
(\sigma^{\tau})^1_5=& -\tfrac{k_e}{2r_e^2}{\scriptstyle(\tau+1)}\, e^1 + \tfrac{k_e}{2r_e^2\sqrt{r_e^2s_e^2-|u_e|^2}}{\scriptstyle\big( \rho r_e^2 (\tau - 1) + (u_{e2}-\lambda r_e^2) (\tau + 1) \big)}\, e^3- \tfrac{k_e\, u_{e1}}{2r_e^2\sqrt{r_e^2s_e^2-|u_e|^2}}{\scriptstyle(\tau+1)}\, e^4\,,\\[3pt]
(\sigma^{\tau})^1_6=&\ \tfrac{k_e}{2r_e^2}{\scriptstyle(\tau+1)}\, e^2 - \tfrac{k_e\, u_{e1}}{2r_e^2\sqrt{r_e^2s_e^2-|u_e|^2}}{\scriptstyle(\tau+1)}\, e^3+ \tfrac{k_e}{2r_e^2\sqrt{r_e^2s_e^2-|u_e|^2}}{\scriptstyle\big( \rho r_e^2 (\tau - 1) - (u_{e2}-\lambda r_e^2) (\tau + 1) \big)}\, e^4\,,\\[3pt]
(\sigma^{\tau})^3_4=& -\tfrac{k_e\,(\lambda\, u_{e1}-r_e^2 y)}{r_e^2s_e^2-|u_e|^2}{\scriptstyle(\tau-1)}\, e^5 -\tfrac{k_e\,(|u_e|^2 -\lambda\,r_e^2 u_{e2} +r_e^4 x)}{r_e^2(r_e^2s_e^2-|u_e|^2)}{\scriptstyle(\tau-1)}\, e^6\,,\\[3pt]
(\sigma^{\tau})^3_5=& -\tfrac{k_e\,(\rho\,r_e^2\, (\tau-1)-u_{e2}\,(\tau+1))}{2r_e^2\sqrt{r_e^2s_e^2-|u_e|^2}}\, e^1 +\tfrac{k_e\,u_{e1} (\tau+1)}{2r_e^2\sqrt{r_e^2s_e^2-|u_e|^2}}\, e^2-\tfrac{k_e\,(|u_e|^2 -\lambda\,r_e^2 u_{e2} +r_e^4 x)}{2r_e^2(r_e^2s_e^2-|u_e|^2)}{\scriptstyle(\tau+1)}\, e^3 \\
& +\tfrac{k_e\,(\lambda\, u_{e1}-r_e^2 y)}{2(r_e^2s_e^2-|u_e|^2)}{\scriptstyle(\tau+1)}\, e^4\,,\\[3pt]
(\sigma^{\tau})^3_6=&\ \tfrac{k_e\,u_{e1} (\tau+1)}{2r_e^2\sqrt{r_e^2s_e^2-|u_e|^2}}\, e^1-\tfrac{k_e\,(\rho\,r_e^2\, (\tau-1)+u_{e2}\,(\tau+1))}{2r_e^2\sqrt{r_e^2s_e^2-|u_e|^2}}\, e^2 +\tfrac{k_e\,(\lambda\, u_{e1}-r_e^2 y)}{2(r_e^2s_e^2-|u_e|^2)}{\scriptstyle(\tau+1)}\, e^3 \\ 
& +\tfrac{k_e\,(|u_e|^2 -\lambda\,r_e^2 u_{e2} +r_e^4 x)}{2r_e^2(r_e^2s_e^2-|u_e|^2)}{\scriptstyle(\tau+1)}\, e^4\,,\\[3pt]
(\sigma^{\tau})^5_6=&0\,,
\end{aligned}
$$
together with the following relations:
$$
\begin{aligned}
&(\sigma^{\tau})^2_3 = - (\sigma^{\tau})^1_4\,, \quad (\sigma^{\tau})^2_4 = (\sigma^{\tau})^1_3\,,\ \   &(\sigma^{\tau})^2_5 = - (\sigma^{\tau})^1_6\,,\quad &(\sigma^{\tau})^2_6 = (\sigma^{\tau})^1_5\,,\\
&(\sigma^{\tau})^4_5 = - (\sigma^{\tau})^3_6\,, \quad (\sigma^{\tau})^4_6 = (\sigma^{\tau})^3_5\,.
\end{aligned}
$$
Finally, by using the curvature 2-forms $(\Omega^\tau)^i_j$ explicitly given in Appendix~A
%Appendix \ref{apendiceA} 
and equation \eqref{traceRmRm}, the result follows from a long but direct calculation.
\end{proof}

\medskip

Let $G$ be a nilpotent Lie group endowed with a left-invariant complex structure $J$ defined by \eqref{J-nilp}. Moreover, let us consider the following left-invariant closed (3,0)-form 
\begin{equation}\label{Psi}
\Psi:=\zeta^1\wedge\zeta^2\wedge\zeta^3.
\end{equation}
Given the Chern connection $\nabla^c$ of the metric $\omega$, i.e. $\tau=1$ in the Gauduchon family, one always has $\nabla^{c}\Psi=0$; whereas, by the proof of the previous proposition, 
one gets

\begin{corollary}\label{cor-traza-tau}
Under the hypotheses of Proposition~\ref{traza-tau}. If $\tau\neq1$, then
$\nabla^{\tau}\Psi=0$ if and only if the Hermitian metric $\omega$ is balanced.
\end{corollary}

\begin{proof}
In terms of an adapted basis $\{ e^l\}_{l=1}^6$ as in Proposition~\ref{adapted-ecus}, we have that 
$$
\nabla^{\tau}((e^1+i\,e^2)\wedge(e^3+i\,e^4)\wedge(e^5+i\,e^6))=0
$$ 
if and only if the connection 1-forms satisfy
$
(\sigma^{\tau})^1_2 + (\sigma^{\tau})^3_4 + (\sigma^{\tau})^5_6 = 0$.
Since the left-invariant (3,0)-forms are related by $\Psi=c\,(e^1+i\,e^2)\wedge(e^3+i\,e^4)\wedge(e^5+i\,e^6)$ 
for some non-zero constant $c$, 
we get that $\nabla^{\tau}\Psi=0$ if and only if 
$$
(\sigma^{\tau})^1_2 + (\sigma^{\tau})^3_4 + (\sigma^{\tau})^5_6 = -\frac{k_e\,(\tau-1)}{r_e^2s_e^2-|u_e|^2} \Big((\lambda\, u_{e1}-r_e^2 y) \,e^5 +(s_e^2 - \lambda\, u_{e2} + r_e^2 x) \,e^6 \Big)=0,
$$
where we have made use of the connection 1-forms given in the proof of Proposition~\ref{traza-tau}. Therefore, given $\tau\not=1$, the above equality holds if and only if
\begin{equation}\label{condis}
\lambda\, u_{e1}-r_e^2 y = 0 = s_e^2 - \lambda\, u_{e2} + r_e^2 x.
\end{equation}

On the other hand, by means of the structure equations in the adapted basis, one directly gets that the metric $\omega$ satisfies the balanced condition $d\omega^2=0$ if and only if 
$$
0=\omega\wedge d\omega = (e^{12}+e^{34})\wedge d e^{56},
$$
which is equivalent to \eqref{condis}.
\end{proof}

\section{The first Anomaly flow equation on nilpotent Lie groups}\label{sec-flow}

\noindent
We now study the behaviour of a general solution to the first equation in the Anomaly flow for our class of nilpotent Lie groups, under certain assumptions. In particular, since we focus on invariant solutions, the first statement in Theorem \ref{thm_A} and Theorem \ref{thm_B} will follow, respectively, by Theorem \ref{almost-diag-prop} and Theorem \ref{bal-cond} below.\medskip

Let $G$ be a 6-dimensional $2$-step nilpotent Lie group with ${b_1\geq 4}$, endowed with a left-invariant non-parallelizable complex structure $J$. Let $\{\zeta^1,\zeta^2,\zeta^3\}$ be a left-invariant $(1,0)$-coframe on $(G,J)$ satisfying \eqref{J-nilp} and let
$$\Psi:=\zeta^1\wedge\zeta^2\wedge\zeta^3$$
be the left-invariant closed (3,0)-form defined in \eqref{Psi}. 

\begin{assumption}\label{condition}
Let $(\omega_t,H_t)$ be the couple of left-invariant Hermitian metrics solving the Anomaly flow \eqref{AF-G} %. Furthermore, given a Gauduchon connection $\nabla$ on $(E,H_t)$, 
 and let ${\rm Tr}(A_t \wedge A_t)$ be a multiple of the (2,2)-form $\zeta^{12\bar1\bar2}$.
\end{assumption}

\begin{remark}\rm Since we are considering left-invariant metrics, the evolution equation of $\omega_t$ reduces to an ODE on the Lie algebra level. We stress that, all the stated results will also hold for the Anomaly flow \eqref{AF-triv} with flat bundle.
\end{remark}

%%%From now on, we assume the holomorphic vector bundle to be
%%%$E:=T^{1,0}G$ and one the following two conditions hold:
%%%\begin{itemize} \setlength\itemsep{0.5em}
%%%\item[1)] the smooth curve of Hermitian metrics $\omega_t$ on $G$ solves
%%%the Anomaly flow \eqref{AF-triv};
%%%\item[2)] the couple of smooth curves of Hermitian metrics $(\omega_t,H_t)$,
%%%with $\omega_t$ on $G$ and $H_t$ along the fibers of $T^{1,0}G$, solves
%%%the Anomaly flow \eqref{AF-G}.
%%%\end{itemize}

%--- Notice that, since the metrics $\omega_t$ and $H_t$ are both left-invariants,
%--- the PDEs describing the Anomaly flows \eqref{AF-triv} and \eqref{AF-G}
%--- reduce to ODEs. \medskip

Let $\omega_t$ be a solution to the Anomaly flow \eqref{AF-G} on $(G,J)$ given by
\begin{equation}\label{F-nilp}
\begin{aligned}
\omega_t  =\ &  \frac{i}{2}\! \left( r(t)^2 \zeta^{1\bar{1}} + s(t)^2 \zeta^{2\bar{2}} + k(t)^2 \zeta^{3\bar{3}} \right)+\frac{1}{2} u(t)\, \zeta^{1\bar{2}} - \frac{1}{2} \overline{u(t)}\, \zeta^{2\bar{1}}\\[5pt]
 &  +\frac{1}{2} v(t)\, \zeta^{2\bar{3}} - \frac{1}{2} \overline{v(t)}\, \zeta^{3\bar{2}}+\frac{1}{2} z(t)\, \zeta^{1\bar{3}} - \frac{1}{2} \overline{z(t)}\, \zeta^{3\bar{1}}\,,
\end{aligned}
\end{equation}
where $r(t),s(t),k(t)$ are positive real functions and $u(t),v(t),z(t)$ are complex functions. Then, we have
%the evolution equation of $\omega_t$ in \eqref{AF-G} reduces to a single ODE equation.

\begin{proposition}\label{AF-evo-F} 
Let Assumption \ref{condition} hold. Then, for every Gauduchon connection $\nabla_t^{\tau}$ on $(G,\omega_t)$, the evolution equation of $\omega_t$ in the Anomaly flow reduces to the ODE
\begin{equation}\label{evol_omega}
\frac d{dt}\,(\Vert\Psi\Vert_{\omega_t}\,\omega_t^2) = K(t,\alpha',\tau)\, \zeta^{12\bar{1}\bar{2}}\,,
\end{equation}
with $K(t,\alpha',\tau)$ depending on the structure equations \eqref{J-nilp} of $(G,J)$ and on the curvature $A_t$ of the connection on $(E,H_t)$.
\end{proposition}

\begin{proof} By means of \eqref{J-nilp} and \eqref{F-nilp}, a direct computation yields that
\begin{equation}\label{dem-ddbar}
2i\partial \bar\partial \omega_t=-k(t)^2( \bar \partial\zeta^3\wedge \partial \zeta^{\bar 3}-\partial \zeta^3\wedge\bar \partial \zeta^{\bar 3})= k(t)^2(\rho+\lambda^2-2x)\zeta^{12\bar1\bar2}\,.
\end{equation}
On the other hand, by means of Proposition \ref{traza-tau}, the trace of the curvature of a Gauduchon connection $\nabla^\tau$ on $(G,\omega_t)$ satisfies 
$$
{\rm Tr}(Rm^{\tau}\wedge Rm^{\tau}) = C(t) \, \zeta^{12\bar{1}\bar{2}}\,.
$$
Therefore, by the assumption on the curvature $A_t$, we have 
$$
\partial_t(\Vert\Psi\Vert_{\omega_t}\,\omega_t^2) = {i\partial\overline\partial \omega_t - \frac{\alpha'}{4}\left( {\rm Tr}(Rm^{\tau}\wedge Rm^{\tau}) - {\rm Tr}(A_t\wedge A_t)\right)}
= K(t,\alpha',\tau)\, \zeta^{12\bar{1}\bar{2}}\,,
$$
where $K(t,\alpha',\tau)$ also depends on $(G,J)$ and $A_t$.
\end{proof}

Let us now analyze equation \eqref{evol_omega} in more detail. A direct computation yields that
$$
\begin{aligned}
2\,\omega_t \wedge \omega_t = &\left( r(t)^2 s(t)^2 \!-\! |u(t)|^2 \right) \zeta^{12\bar{1}\bar{2}} + \left( r(t)^2 k(t)^2 \!-\! |z(t)|^2 \right) \zeta^{13\bar{1}\bar{3}} + \left( s(t)^2 k(t)^2 \!-\! |v(t)|^2 \right) \zeta^{23\bar{2}\bar{3}}\\
& - i \left( r(t)^2 v(t) - i\, z(t)\overline{u(t)} \right) \zeta^{12\bar{1}\bar{3}} 
+ i \left( r(t)^2 \overline{v(t)} + i\, u(t) \overline{z(t)} \right) \zeta^{13\bar{1}\bar{2}}\\
&+ i \left( s(t)^2 z(t) + i\,u(t) v(t) \right) \zeta^{12\bar{2}\bar{3}} 
- i \left( s(t)^2 \overline{z(t)} - i\, \overline{u(t)} \overline{v(t)} \right) \zeta^{23\bar{1}\bar{2}} \\
& - i \left( k(t)^2 u(t) - i\, z(t) \overline{v(t)} \right) \zeta^{13\bar{2}\bar{3}} 
+ i \left( k(t)^2 \overline{u(t)} + i\, v(t) \overline{z(t)} \right)\zeta^{23\bar{1}\bar{3}}\,.
\end{aligned}
$$
Therefore, by substituting in \eqref{evol_omega}, one gets that the following relations hold along the flow:
\begin{equation}\label{ecu-main}
\frac d{dt}\left( \Vert\Psi\Vert_{\omega_t} (r(t)^2 s(t)^2 \!-\! |u(t)|^2) \right)=2\, K(t,\alpha',\tau),
\end{equation}
and
\begin{equation}\label{ecu-1}
\frac d{dt}\left( \Vert\Psi\Vert_{\omega_t} (r(t)^2 k(t)^2 \!-\! |z(t)|^2) \right)=0 \quad\Longrightarrow\quad   r(t)^2 k(t)^2 \!-\! |z(t)|^2 =\displaystyle{\frac{c_1}{\Vert\Psi\Vert_{\omega_t}}}\,,
\end{equation}
\begin{equation}\label{ecu-2}
\frac d{dt}\left( \Vert\Psi\Vert_{\omega_t} (s(t)^2 k(t)^2 \!-\! |v(t)|^2) \right)=0
\quad\Longrightarrow\quad   s(t)^2 k(t)^2 \!-\! |v(t)|^2 =\displaystyle{\frac{c_2}{\Vert\Psi\Vert_{\omega_t}}}\,,
\end{equation}
\begin{equation}\label{ecu-3}
\frac d{dt}\left( \Vert\Psi\Vert_{\omega_t} (r(t)^2 v(t) - i\, z(t)\overline{u(t)} ) \right)=0
\quad\Longrightarrow\quad   r(t)^2 v(t) - i\, z(t)\overline{u(t)}  =\displaystyle{\frac{c_3}{\Vert\Psi\Vert_{\omega_t}}}\,,
\end{equation}
\begin{equation}\label{ecu-4}
\frac d{dt}\left( \Vert\Psi\Vert_{\omega_t} (s(t)^2 z(t) + i\, u(t) v(t) ) \right)=0\quad\Longrightarrow\quad   s(t)^2 z(t) + i\, u(t) v(t) =\displaystyle{\frac{c_4}{\Vert\Psi\Vert_{\omega_t}}}\,,
\end{equation}
\begin{equation}\label{ecu-5}
\frac d{dt}\left( \Vert\Psi\Vert_{\omega_t} (k(t)^2 u(t) - i\, z(t) \overline{v(t)} ) \right)=0\quad\Longrightarrow\quad   k(t)^2 u(t) - i\, z(t) \overline{v(t)} =\displaystyle{\frac{c_5}{\Vert\Psi\Vert_{\omega_t}}}\,,
\end{equation}
for some constants $c_1,c_2\in\RR$ with $c_1,c_2>0$, and $c_3,c_4,c_5\in\C$, which are determined by the initial metric $\omega_0$.\medskip

\subsection{Special Hermitian metrics along the flow}\hfill\medskip

\noindent
In \cite{PPZ18} Phong, Picard and Zhang proved that the Anomaly flow preserve the conformally balanced condition, once the connections $\nabla^\tau$ and $\nabla^\kappa$ are both Chern. In the following, we extend such a result to any connection in the Gauduchon family for our class of nilpotent Lie groups. Moreover, we also show that the locally conformally K\"ahler condition is preserved along the flow.\smallskip

A Hermitian metric $\omega$ is said to be \emph{locally conformally K\"ahler} if it is conformal to some local K\"ahler metric in a neighborhood of each point. Recall that such metrics are also characterized by the condition $d\omega=\theta\wedge\omega$, where $\theta$ is the (closed) Lee form. 

\begin{theorem}\label{bal-cond}
Under Assumption~\ref{condition}, we have:
\begin{itemize}\setlength\itemsep{0.5em}
\item[(i)] If $\omega_0$ is balanced, then $\omega_t$ remains balanced along the Anomaly flow.
\item[(ii)] If $\omega_0$ is locally conformally K\"ahler, then $\omega_t$ remains locally conformally K\"ahler along the Anomaly flow. 
\end{itemize}
\end{theorem}

\begin{remark}\rm By \cite[Theorem 1.2]{FPS04}, the pluriclosed condition $\partial\db \omega=0$ for a left-invariant metric $\omega$ on $(G,J)$ 
only depends on the complex structure $J$ (see also \cite{U}). Therefore, if the initial metric $\omega_0$ is pluriclosed, the solution $\omega_t$ to the Anomaly flow holds pluriclosed.
\end{remark}

\begin{proof}[Proof of Theorem \ref{bal-cond}] 
In view of \cite[Proposition 25]{U}, the left-invariant Hermitian metric  $\omega_t$ 
%%%$$
%%%2\, \omega=i \left(r^2\zeta^{1\bar{1}} + s^2\zeta^{2\bar{2}} %%%+k^2\zeta^{3\bar{3}}\right)+u\,\zeta^{1\bar{2}}-\bar{u}\,%%%\zeta^{2\bar{1}} +v\,\zeta^{2\bar{3}}-\bar{v}\,\zeta^{3\bar{2}}+z\,%%%\zeta^{1\bar{3}}-\bar{z}\,\zeta^{3\bar{1}}
%%%$$
%%%on $G$ (with $\{\zeta^1,\zeta^2,\zeta^3\}$ satisfying \eqref{J-nilp}) 
is balanced if and only if
\begin{equation}\label{Bnilp-condition}
s(t)^2k(t)^2-|v(t)|^2 + (x+i\, y)\left( r(t)^2k(t)^2-|z(t)|^2 \right) = i\,\lambda\left( k(t)^2\,\overline{u(t)} +i\, v(t)\overline{z(t)}\right).
\end{equation}
On the other hand, by means of \eqref{ecu-1} and \eqref{ecu-2}, the left-hand side of \eqref{Bnilp-condition} reduces to
$$
s(t)^2k(t)^2-|v(t)|^2 +  (x+i\, y)\left( r(t)^2k(t)^2-|z(t)|^2 \right) = {\frac{c_1 (x+i\, y)+c_2}{\Vert\Psi\Vert_{\omega_t}}},
$$
while, by means of \eqref{ecu-5}, the right-hand side of \eqref{Bnilp-condition} is equal to
$$
i\,\lambda\left( k(t)^2\,\overline{u(t)} +i\, v(t)\overline{z(t)}\right) = {\frac{i\,\lambda\, {\bar c_5}}{\Vert\Psi\Vert_{\omega_t}}}.
$$
Thus, $\omega_t$ is a balanced metric if and only if 
$c_1 (x+i\, y)+c_2=i\, \lambda\, {\bar c_5}$. Since the constants $c_1,c_2$ and $c_5$ only depend on the initial metric $\omega_0$, it follows that $\omega_t$ satisfies the balanced condition if and only if $\omega_0$ does.

Let us now prove (ii). By means of \cite[Proposition 32]{U}, if $\omega_0$ is locally conformally K\"ahler, we must have
$$
\rho=\lambda=y=0\qquad \text{and}\qquad x=1
$$ 
in the complex structure equations \eqref{J-nilp}.
Moreover, $\omega_t$ is locally conformally K\"ahler if and only if 
$$
r(t)^2 k(t)^2 \!-\! |z(t)|^2 = s(t)^2 k(t)^2 \!-\! |v(t)|^2 
\quad\mbox{ and }\quad
k(t)^2 u(t)=i\, z(t) \overline{v(t)}.
$$
Therefore, in view of \eqref{ecu-1}, \eqref{ecu-2} and \eqref{ecu-5}, it follows that $\omega_t$ is a locally conformally K\"ahler metric if and only if 
$c_1-c_2=c_5=0$. Finally, since the constants $c_1,c_2,c_5$ only depend on the initial metric $\omega_0$, we get that $\omega_t$ is locally conformally K\"ahler if and only if $\omega_0$ is locally conformally K\"ahler as well, and the claim follows.
\end{proof}

%\begin{remark}\rm In view of Theorem \ref{bal-cond}, the balanced condition on $\omega_t$ does not depend on the coefficients $c_3$ and $c_4$. 
%\end{remark}

%\begin{example}\rm Let $G$ be a 2-step nilpotent Lie group arising from \eqref{J-nilp}. If we consider $\lambda=1$ and $D\in\RR$ such that $0\leq D<\frac{1}{4}$, then there are no left-invariant diagonal balanced metric on $G$, since the diagonal condition directly implies $c_5=0$. Nonetheless, under these assumptions, it is always possible to find an almost diagonal balanced metric on $M$.
%\end{example}

\subsection{Reduction to almost diagonal initial metrics and the general solution}\hfill\medskip

\noindent In the following, we prove that any initial Hermitian metric $\omega_0$ can be taken to be almost diagonal. Then, we use this result to obtain the general solution $\omega_t$ to the first evolution equation in the Anomaly flow \eqref{AF-G}.\smallskip

A Hermitian metric $\omega$ is said to be \emph{almost diagonal} if its metric coefficients satisfy $v=z=0$ in \eqref{2forma}. Our next result shows that we can always choose a preferable $(1,0)$-coframe on $(G,J)$ such that the metric $\omega$ is almost diagonal.

\begin{lemma}\label{almost-diag-lemma}
Let $\omega$ be a left-invariant Hermitian metric on $(G,J)$. Then, there exists an automorphism which preserves both the complex structure equations \eqref{J-nilp} and the (3,0)-form $\Psi$, and such that $\omega$ reduces to an almost diagonal form.
\end{lemma}

\begin{proof} 
To prove this lemma we essentially use the same argument as in Proposition \ref{adapted-ecus}. Let us consider the automorphisms induced by \eqref{J-nilp-sigma-anterior}. By construction this automorphism preserves the complex structure equations (see \eqref{J-nilp-sigma}). Moreover, a direct computation yields that
$$\Psi=\zeta^1\wedge\zeta^2\wedge\zeta^3=\sigma^1\wedge\sigma^2\wedge\sigma^3$$
and hence the holomorphic (3,0)-form $\Psi$ is also preserved. 
Finally, in terms of the coframe $\{\sigma^l\}_{l=1}^3$ the Hermitian metric $\omega$ expresses as in
\eqref{J-nilp-sigma-posterior} and hence the claim follows.
\end{proof}

We are now in a position to describe the general solution to the Anomaly flow starting from an almost diagonal metric.

\begin{theorem}\label{almost-diag-prop}
Under Assumption~\ref{condition}, the Anomaly flow preserves the almost diagonal condition. More concretely, if $\omega_0$ is almost diagonal, then $\omega_t$ evolves as
\begin{equation}\label{almost-diagonal-Ft}
\omega_t =\frac{i}{2}\left(r(t)^2 \zeta^{1\bar{1}} + \frac{c_2}{c_1}\,r(t)^2 \zeta^{2\bar{2}} + \frac{c_1c_2\!-\!|c_5|^2}{8} \zeta^{3\bar{3}}\right)+\frac{c_5}{2c_1}\,r(t)^2\, \zeta^{1\bar{2}} - {\frac{\bar{c}_5}{2c_1}}\,r(t)^2\, \zeta^{2\bar{1}},
\end{equation}
where $c_1,c_2>0$ and $c_5\in\C$ with $c_1c_2>|c_5|^2$, 
are constants determined by the initial metric $\omega_0$.
Furthermore,
$$
\Vert\Psi\Vert_{\omega_t}=\frac{8\, c_1}{(c_1c_2\!-\!|c_5|^2)\,r(t)^2}\,.
$$
\end{theorem}

%\begin{theorem}\label{almost-diag-prop}
%Let the initial metric $\omega_0$ be almost diagonal (i.e. $v(0)=z(0)=0$). Then, the solution $\omega_t$ to the Anomaly flow \eqref{AF-G} remains almost diagonal, in the defining interval of $(\omega_t,H_t)$. Moreover, it satisfies
%\begin{equation}\label{almost-diagonal-Ft}
%\omega_t \!=\! \displaystyle{\frac{i}{2}\,r(t)^2 \omega^{1\bar{1}}} + \displaystyle{\frac{ic_2}{2c_1}\,r(t)^2 \omega^{2\bar{2}} + 
%\frac{i(c_1c_2\!-\!|c_5|^2)}{16} \omega^{3\bar{3}}}
%+\displaystyle{\frac{c_5}{2c_1}}\,r(t)^2\, \omega^{1\bar{2}} - \displaystyle{\frac{\overline{c_5}}{2c_1}}\,r(t)^2\, \omega^{2\bar{1}},
%\end{equation}
%where $c_1,c_2>0$ and $c_5\in\C$ satisfy $c_1\,c_2>|c_5|^2$, and%which in turn implies
%$$
%\Vert\Psi\Vert_{\omega_t}=\displaystyle{\frac{8\, c_1}{(c_1\,c_2-|c_5|^2)\,r(t)^2}}\,.
%$$
%\end{theorem}

\begin{proof}
Since equations \eqref{ecu-3} and \eqref{ecu-4} hold, the functions $v(t)$ and $z(t)$ satisfy
$$
v(t)={\frac{c_3\, s(t)^2+i\,c_4\, \overline{u(t)} }{\Vert\Psi\Vert_{\omega_t} (r(t)^2 s(t)^2 - |u(t)|^2)}}\,, \qquad z(t)={\frac{-i\,c_3\, u(t) +c_4\, r(t)^2}{\Vert\Psi\Vert_{\omega_t} (r(t)^2 s(t)^2 - |u(t)|^2)}}\,,
$$
for any $t$ in the defining interval. On the other hand, since we assumed $\omega_0$ to be almost diagonal, i.e. $v(0)=z(0)=0$, we get
$c_3=c_4=0$. Therefore, $v(t)=0$ and $z(t)=0$ and hence the solution $\omega_t$ holds almost diagonal.

Let us now prove the second part of the statement. As a direct consequence of \eqref{ecu-1}, \eqref{ecu-2} and \eqref{ecu-5}, it follows that
$$
(r(t)^2 s(t)^2-|u(t)|^2) k(t)^4=\displaystyle{\frac{c_1c_2-|c_5|^2}{\Vert\Psi\Vert_{\omega_t}^2}}\,,
$$
which implies 
$$
\Vert\Psi\Vert_{\omega_t}^2 = \displaystyle{\frac{c_1c_2-|c_5|^2}{(r(t)^2 s(t)^2-|u(t)|^2) k(t)^4}}\,,
$$
with $c_1c_2-|c_5|^2>0$ by the  positive definiteness of the metric $\omega_0$. Moreover, by the definition of $\Vert\Psi\Vert_{\omega_t}^2$, we have
$$
\Vert\Psi\Vert_{\omega_t}^2%=\displaystyle{i\, \Omega\wedge \overline\Omega \left( \frac{\omega_t^3}{3!} \right)^{-1}}
=\frac{1}{\det \omega_t}
={\frac{8}{(r(t)^2 s(t)^2-|u(t)|^2) k(t)^2}}
$$
and hence
$$
k(t)=\sqrt{\frac{c_1\,c_2-|c_5|^2}{8}}.
$$
In particular, $k(t)$ is constant. Finally, by means of \eqref{ecu-1}, \eqref{ecu-2} and \eqref{ecu-5}, we have
$$
0=c_2\, r(t)^2 k(t)^2 - c_1\, s(t)^2 k(t)^2 = (c_2\, r(t)^2 - c_1\, s(t)^2) k(t)^2
$$
and
$$
0=c_5\, r(t)^2 k(t)^2 - c_1\, u(t) k(t)^2 = (c_5\, r(t)^2 - c_1\, u(t)) k(t)^2\,,
$$
which respectively imply 
$$
s(t)^2 = \frac{c_2}{c_1}\, r(t)^2\qquad\text{and}\qquad u(t) = \frac{c_5}{c_1}\, r(t)^2\,,
$$ 
and the claim follows.
\end{proof}

When the initial metric $\omega_0$ is {\em diagonal} (that is, $u(0)\!=\!v(0)\!=\!z(0)\!=\!0$), the above result simplifies~to

\begin{corollary}\label{cor-diag}
Under Assumption~\ref{condition}, the Anomaly flow preserves the diagonal condition. Specifically, if $\omega_0$ is diagonal, then $\omega_t$ is given by
\begin{equation*}
\omega_t= \frac{i}{2}\left(\,r(t)^2 \zeta^{1\bar{1}} + \displaystyle{\frac{c_2}{c_1}\,r(t)^2 \zeta^{2\bar{2}} + \frac{c_1c_2}{8} \zeta^{3\bar{3}}}\right)\,,
\end{equation*}
where $c_1=\frac{\sqrt{8}\,r(0)k(0)}{s(0)}>0$ and $c_2=\frac{\sqrt{8}\,s(0)k(0)}{r(0)}>0$. Moreover, 
$\Vert\Psi\Vert_{\omega_t}={\frac{8}{c_2\,r(t)^2}}$.
\end{corollary}

%We now compute the trace of the curvature of the Gauduchon connections of the almost diagonal metrics along the Anomaly flow.
Our next result describes the evolution of the trace ${\rm Tr}(Rm^{\tau}_t\wedge Rm^{\tau}_t)$ along the Anomaly flow, under the assumption for the initial metric $\omega_0$ to be almost diagonal.

\begin{proposition}\label{traza-tau-almost-diag}
Under the hypotheses of Theorem~\ref{almost-diag-prop}, the trace of the curvature of the Gauduchon connection $\nabla^{\tau}$ of $(G,\omega_t)$ satisfies 
$$
{\rm Tr}(Rm^{\tau}_t\wedge Rm^{\tau}_t) = \frac{C}{r(t)^4} \, \zeta^{12\bar{1}\bar{2}},
$$
where $C=C(\omega_0,\tau)$ 
%%%$C=C(\omega_0,\tau)=C(\rho,\lambda,x,y;c_1,c_2,c_5;\tau)$
is a constant depending only on the initial metric $\omega_0$ and the connection $\nabla^\tau$. 
\end{proposition}

\begin{proof}
The result is a consequence of Proposition~\ref{traza-tau}. Indeed, the coefficients $r_e,s_e,k_e$ and ${u_e=u_{e1}+i\,u_{e2}}$ appearing in Proposition \ref{traza-tau} are related to the coefficients of the metric $\omega_t$ via~\eqref{new-coeffs-1}. On the other hand, by means of Theorem~\ref{almost-diag-prop}, the metric $\omega_t$ holds almost diagonal and by
\eqref{new-coeffs-1} and \eqref{almost-diagonal-Ft} we get
\begin{equation}\label{coeff_alm}
r_e^2=r(t)^2,\quad s_e^2=\frac{c_2}{c_1} r(t)^2,\quad k_e^2=\frac{c_1c_2-|c_5|^2}{8},\quad u_e=\frac{c_5}{c_1} r(t)^2.
\end{equation}
Therefore, the claim follows by \eqref{coeff_alm} and the formula of the trace given in Proposition~\ref{traza-tau}.
%, since ${\rm Tr}(Rm_t^{\tau}\wedge Rm_t^{\tau})$ is a multiple of $\zeta^{12\bar{1}\bar{2}}$ and the $\frac{C}{r(t)^4}$, where $C$ does not depend on $t$ but only on $\omega_0$ and $\tau$.
\end{proof}

\section{The Anomaly flow with flat holomorphic bundle}\label{sec-model}

\noindent
We now focus on the Anomaly flow~\eqref{AF-triv}. In particular, we show that this flow always reduces to a single ODE, which we call {\em model problem}, and the second part of Theorem \ref{thm_A} will directly follow by Theorem \ref{model-prob} below. Moreover, the qualitative behaviour of the model problem will be investigated.\medskip

Let $G$ be a 6-dimensional $2$-step nilpotent Lie group with ${b_1\geq 4}$, equipped with a left-invariant non-parallelizable complex structure $J$. Let $\{\zeta^1,\zeta^2,\zeta^3\}$ be a left-invariant $(1,0)$-coframe on $(G,J)$ satisfying \eqref{J-nilp} and let
$\Psi$
be the left-invariant closed (3,0)-form defined in \eqref{Psi}, i.e. $\Psi:=\zeta^1\wedge\zeta^2\wedge\zeta^3$.

In view of Proposition \ref{AF-evo-F} and Theorem \ref{almost-diag-prop},  the coefficient $r(t)^2$ of the metric $\omega_t$ in \eqref{almost-diagonal-Ft} evolves as 
\begin{equation*}\label{model-prim}
\partial_t\, r(t)^2 =\frac{c_1}{4}\, K(t,\alpha',\tau)\,,
\end{equation*}
where the right-hand side is given by
$$
K(t,\alpha',\tau)\,\zeta^{12\bar 1\bar 2}= i \partial\overline\partial \omega_t - \frac{\alpha'}{4} {\rm Tr}(\Omega_t^\tau\wedge \Omega_t^\tau)\,.
$$
On the other hand, by means of \eqref{dem-ddbar} and Theorem \ref{almost-diag-prop}, we get 
$$
i\partial\overline\partial \omega_t= B\,\zeta^{12\bar 1\bar 2}\,, 
$$
for the constant $B=B(\omega_0)=\frac{c_1c_2\!-\!|c_5|^2}{16}(\rho+\lambda^2-2x)\in\RR$; while, by means of Proposition \ref{traza-tau-almost-diag}, we have 
$$
{\rm Tr}(\Omega_t^\tau\wedge \Omega_t^\tau) = \frac {C}{r(t)^4}\,\zeta^{12\bar 1\bar 2}\,,
$$
for a constant $C=C(\omega_0,\tau)\in \RR$. Therefore, we get
$$
\frac{c_1}{4}\, K(t,\alpha',\tau) = K_1+\frac{ K_2}{r(t)^4},
$$
where $K_1=\frac{c_1\, B}{4}$ and $K_2=- \alpha'\,\frac{c_1\,C}{16}$, and the following theorem holds.

\begin{theorem} \label{model-prob}
The Anomaly flow \eqref{AF-triv} is equivalent to the {\em model problem}
\begin{equation}\label{model}
\frac d{dt}\, r(t)^2= K_1+\frac{ K_2}{r(t)^4}\,, 
\end{equation}
where $ K_{1}, K_2\in\RR$ are constants depending on $K_1=K_1(\omega_0)$ and $K_2=K_2(\omega_0,\alpha',\tau)$. 
\end{theorem}

\subsection{Qualitative behaviour of the model problem}\label{qualit-model-prob}\hfill\medskip

\noindent We now investigate the qualitative behaviour of the model problem \eqref{model}, which can be rewritten as
\begin{equation}\label{model-ref}
h'(t) = K_1+\frac{K_2}{h(t)^2}\,,\qquad h(t)> 0\,.
\end{equation}
A solution $h(t)$ to \eqref{model-ref} is said to be {\em immortal}, {\em eternal} or {\em ancient} if the defining interval $(T^-,T^+)$ is equal to $(-\varepsilon,+\infty)$, $(-\infty,+\infty)$ or $(-\infty,\varepsilon)$ for some $\varepsilon>0$, respectively.\medskip

When either $K_1=0$ or $K_2=0$ the ODE \eqref{model-ref} can be explicitly solved, otherwise we work as follows.

\subsubsection*{$\bullet$ $K_1>0$ and $K_2>0$} \hfill\medskip

\noindent Under these assumptions, \eqref{model-ref} does not admit any stationary point. Nonetheless, we have the following 

\begin{proposition}\label{pos-pos}
Any solution $h(t)$ to the model problem \eqref{model-ref} is immortal. In particular, $h(t)\sim K_1\cdot t$ as $t\to +\infty$.
\end{proposition}

\begin{proof}
Let $h(t)$ be a solution to the model problem \eqref{model-ref}. Since
$$
h'(t)= K_1 + \frac{K_2}{h(t)^2}>0\,,
$$
it follows that $h(t)\geq h(0)$, for every $t\in[0,T^+)$. On the other hand,
$$
h'(t)\leq K_1 + \frac{K_2}{h(0)^2}
$$
and the long-time existence follows, since $h(t)\leq c\, t+h(0)$ with $c:=K_1 + \frac{K_2}{h(0)^2}$. 

Let us now suppose by contradiction that $h'(t)\to 0$ as $t\to+ \infty$. Then, this would imply
$$
\lim_{t\to \infty}\left( K_1+\frac{K_2}{h(t)^2}\right)=0\,,
$$
which is not possible since $K_1,K_2>0$. Therefore, we have
$$
\lim_{t\to \infty} h'(t) =K_1
$$
and hence $h(t)\sim K_1\cdot t$ as $t\to +\infty$. Finally, a similar argument shows that if the solution exists backward in time for any $t<0$, then 
$$h(t)\sim K_1\cdot t\,,\quad \text{as}\,\,\, t\to-\infty\,,$$ 
which is not possible since $h(t)>0$.
\end{proof}

\subsubsection*{$\bullet$ $K_1>0$ and $K_2<0$} \hfill\medskip

\noindent Let us denote by $h_0:=\sqrt{-K_2/K_1}$. Then, we have 

\begin{proposition}\label{pos-neg}
Let $h(t)$ be a solution to the model problem \eqref{model-ref}. It follows that
\begin{itemize}\setlength\itemsep{0.5em}
\item[\rm (i)] if $h(0)=h_0$, then the solution is stationary;
\item[\rm (ii)] if $h(0)>h_0$, then the solution is eternal and $h(t)\sim K_1\cdot t$ as $t\to +\infty$;
\item[\rm (iii)] if $h(0)<h_0$, then the solution is ancient.
\end{itemize}
Furthermore, any solution detects the stationary point as $t\to -\infty$.
\end{proposition}

\begin{proof} Let $h(t)$ be the solution to \eqref{model-ref}. Then, a direct computation yields that $h_0$ is the unique stationary point to the flow, and the first claim follows.

Now, let us suppose $h(0)>h_0$. Then, there exists $\varepsilon>0$ such that ${h(0)=\sqrt{-\tfrac{K_2}{K_1}+\varepsilon}}$. Therefore, we have
$$
h'(0) = \frac{\varepsilon K_1^2}{-K_2+\varepsilon K_1}>0\,,
$$
and hence $h(t)'>0$ for every $t\in(T^-,T^+)$. On the other hand,
$$
h'(t) \leq K_1\quad \Longrightarrow \quad h(t)\leq K_1\,t+h(0)\,, \quad \text{for any}\,\,\, t\geq0\,,
$$
and the long-time existence follows. Moreover, since $h(t)$ is always increasing and $h_0$ is the unique stationary point to the flow, it follows $h(t)\to h_0$ as $t\to-\infty$. Thus, the solution $h(t)$ is eternal. Finally, let us assume by contradiction that $h'(t)\to 0$ as $t\to +\infty$. Then, this would be equivalent to require
$$
\lim_{t\to \infty} K_1+\frac{K_2}{h(t)^2}=0\,,
$$
which is not possible since $h(0)>h_0$, and hence
$$
\lim_{t\to \infty} h'(t) =K_1
$$
proves the second claim.

Now, let us assume ${h(0)=\sqrt{-\frac{K_2}{K_1}-\varepsilon}}<h_0$ for some $\varepsilon>0$. Then, a direct computation yields that
%\begin{equation}\label{dis-ex1}
$$
h'(0) = \frac{-\varepsilon K_1^2}{-K_2+\varepsilon K_1}<0\,,
$$
%\end{equation}
which in turn implies $h'(t)<0$ for every $t\in(T^-,T^+)$. On the other hand, it follows 
$$
h(t)\leq \frac{-\varepsilon K_1^2}{-K_2+\varepsilon K_1}\,t+h(0)\,,\qquad \text{for any $t\geq0$}\,,
$$
and hence $T^+<+\infty$. Moreover, since $h(t)$ is decreasing, we have
%\begin{equation}\label{lim-ex1}
$$
\lim_{t\to T^+} h'(t) = \lim_{t\to T^+} \left(K_1+\frac{K_2}{h(t)^2}\right)=-\infty\,.
$$
%\end{equation}
Finally, since $h(t)$ is always decreasing and there exists a unique stationary solution $h_0$ to the flow, we have that $h(t)\to h_0$ as $t\to-\infty$ and the last claim follows.
\end{proof}

%\begin{remark}\rm As far as the authors know, the Anomaly flow is the second example of a metric flow admitting invariant solutions with both $T^+<+\infty$ and $T^+=+\infty$ on the same homogeneous space (the first one has been found in \cite{AL}, in the context of the pluriclosed flow).
%\end{remark}

\subsubsection*{$\bullet$ $K_1<0$ and $K_2<0$} \hfill\medskip

\noindent Under these assumptions, we have

\begin{proposition}\label{neg-neg}
Any solution $h(t)$ to \eqref{model-ref} is ancient. In particular, $h(t)\sim -K_1\cdot t$ as $t\to -\infty$.
\end{proposition}

The proof of this result can be easily recovered using the same arguments as in Proposition \ref{pos-pos}.

\subsubsection*{$\bullet$ $K_1<0$ and $K_2>0$} \hfill\medskip

\noindent Arguing in the same way of Proposition \ref{pos-neg}, we get

\begin{proposition}\label{neg-pos}
Let $h(t)$ be a solution to \eqref{model-ref}. It follows that
\begin{itemize}\setlength\itemsep{0.5em}
\item[\rm (i)] if $h(0)=h_0$, then the solution is stationary;
\item[\rm (ii)] if $h(0)>h_0$, then the solution is eternal and $h(t)\sim -K_1\cdot t$ as $t\to -\infty$;
\item[\rm (iii)] if $h(0)<h_0$, then the solution is immortal.
\end{itemize}
Furthermore, any solution detects the stationary point as $t\to +\infty$.
\end{proposition}

\subsection{The sign of $K_1$ and its relation to the Fu-Wang-Wu conformal invariant}\label{signosdeKs}\hfill\medskip

\noindent 
We now investigate the relation between the constant $K_1$ appearing in the model problem \eqref{model} and the conformal invariant of $\omega_0$ introduced and studied by Fu, Wang and Wu in \cite{FWW}. We also study the sign of $K_1$ in our class of nilpotent Lie groups.\smallskip

Let $X$ be a compact $n$-dimensional complex manifold and $\omega$ a Hermitian metric on $X$. In \cite{FWW}, the notion of Gauduchon metric has been generalized by the so-called \emph{$k$-th Gauduchon equation}
$$
\partial \overline \partial \omega^k \wedge \omega^{n - k -1} = 0\,,\qquad 1\leq k\leq n-1\,.
$$
Then, since the $k$-th Gauduchon equation may not admit a solution, Fu, Wang and Wu considered the equation (in the conformal class of $\omega$) given by
\begin{equation}\label{gen_gaud}
\frac{i}{2}\, \partial \overline \partial (e^v \omega^k) \wedge \omega^{n - k -1} = \gamma_k (\omega)\, e^v \omega^n\,,\qquad 1\leq k\leq n-1\,,
\end{equation}
proving that there always exist a unique constant $\gamma_k (\omega)$ and a function $v \in {\mathcal C}^{\infty} (X)$ (unique up to a constant) satisfying \eqref{gen_gaud}.
%%%In particular, given a K\"ahler metric $\omega$ on $X$, it follows that
%%%$\gamma_k(\omega) =0$ and $v$ is constant on $X$ for any $1 \leq k \leq n - 1$.
Moreover, the constant $\gamma_k(\omega)$ is invariant under biholomorphisms and it smoothly depends on the metric~$\omega$, and its sign is invariant in the conformal class of~$\omega$ \cite{FWW}.\smallskip

Now, let $(X,\omega)$ be a compact non-K\"ahler Hermitian manifold. In view of \cite[Lemma 3.7]{IP} and \cite[Proposition 3.8]{IP}, for any $1\leq k\leq n-2$ it follows that\smallskip
\begin{itemize}
\item[\rm (i)] if $\omega$ is balanced, then the constant $\gamma_k (\omega)>0$;\medskip
\item[\rm (ii)] if $\omega$ is locally conformally K\"ahler, then the constant $\gamma_k (\omega)<0$.
\end{itemize}
Therefore, we can apply these results to compute the sign of $K_1=K_1(\omega_0)$ in the model problem~\eqref{model}. Indeed, by \cite[Proposition 2.7]{LUV}, any left-invariant Hermitian metric $\omega_0$ on $(G,J)$ given by \eqref{2forma} satisfies 
$$
\frac{i}{2}\,\partial\db \omega_0 \wedge \omega_0
= \frac{k_0^4}{8i\,\det\omega_0} \left(\rho+\lambda^2 - 2\,x \right)\, \omega_0^3\,,
$$
and hence
$$
\gamma_1(\omega_0)=\frac{k_0^4}{8i\,\det\omega_0} \left(\rho+\lambda^2 - 2\,x \right)\,.
$$ 
On the other hand, since $K_1(\omega_0)=\frac{c_1\,k_0^2}{8}(\rho+\lambda^2-2x)$ in the model problem \eqref{model}, we get 
$$
K_1(\omega_0)=\frac{c_1\,i\, \det\omega_0}{k_0^2}\, \gamma_1(\omega_0).
$$
In particular, the sign of $K_1(\omega_0)$ is equal to the one of $\gamma_1(\omega_0)$, which is an invariant of the conformal class of $\omega_0$. Actually, in our context we have that
$$
{\rm sign}\, K_1={\rm sign}\, (\rho+\lambda^2-2x),
$$
and hence it only depends on the complex structure $J$. This fact for $\gamma_1$ was first noticed in 
\cite{FU}. Moreover, since an invariant Hermitian metric on a complex nilmanifold of complex dimension 3 is $1$-st Gauduchon if and only if it is pluriclosed \cite[Proposition 3.3]{FU}, we have the following proposition.

\begin{proposition} 
The sign of $K_1$ in the model problem \eqref{model} only depends on the complex structure~$J$ on $G$. Moreover: 
\begin{itemize}\setlength\itemsep{0.5em}
\item[(i)] If $\omega_0$ is balanced, then $K_1>0$.
\item[(ii)] If $\omega_0$ is locally conformally K\"ahler, then $K_1<0$. 
\item[(iii)] The metric $\omega_0$ is pluriclosed if and only if $K_1=0$.
\end{itemize}
\end{proposition} 
%We mention that the sign of $K_1$ was studied in \cite{LUV} for general 6-dimensional nilmanifolds with invariant complex structure.\smallskip

%Note that the converse of (i) or (ii) do not hold. 
In Table~\ref{table-sign}, we provide the classification of the Lie groups admitting a complex structure satisfying~\eqref{J-nilp}, together with the sign of the invariant~$K_1$.

The first column of Table~\ref{table-sign} describes the nilpotent Lie algebra associated to the Lie group.
Here we use the notation for which the algebras are named as $\frn_k$ and then described (see e.g. \cite{U}).
Moreover, we denote by $N_k$ the Lie group corresponding to $\frn_k$ (second column of the table). About the other columns,
we use the following convention.
The symbol~``${\checkmark}$'' means that the sign of~$K_1$ is the one given in the table for \emph{any} complex structure on the corresponding Lie group $N_k$,
whereas~``${\checkmark}\!\!\!\! \ _{(J)}$'' means that there exist complex structures~$J$ on $N_k$ such that the sign of~$K_1$ is the one described by the column. Therefore, different complex structures may lead to different sings on the same group $N_k$.
Finally, we use~``$-$'' to denote that there are no 
complex structures of the given sign.

\medskip

\begin{table}[!ht]
\renewcommand{\arraystretch}{1.2}
%\resizebox{\textwidth}{!}{
\begin{tabular}{|l|c|c|c|c|}
\hline
%%%\cite{FU}
Lie algebra & Lie group  & \ $K_1<0$ \ & \ \ $K_1=0$ \ \ & \ $K_1>0$ \  \\
\hline
$\frn_{2}=(0,0,0,0,12,34)$ & $N_2$ & $\checkmark\!\!\!\! \ _{(J)}$ & $\checkmark\!\!\!\! \ _{(J)}$ & $\checkmark\!\!\!\! \ _{(J)}$  \\
\hline
$\frn_{3}=(0,0,0,0,0,12\!+\!34)$ & $N_3$ & $\checkmark\!\!\!\! \ _{(J)}$ & $-$ & $\checkmark\!\!\!\! \ _{(J)}$  \\
\hline
$\frn_{4}=(0,0,0,0,12,14\!+\!23)$ & $N_4$ & $\checkmark\!\!\!\! \ _{(J)}$ & $\checkmark\!\!\!\! \ _{(J)}$ & $\checkmark\!\!\!\! \ _{(J)}$  \\
\hline
$\frn_{5}=(0,0,0,0,13\!+\!42,14\!+\!23)$ & $N_5$ & $\checkmark\!\!\!\! \ _{(J)}$ & $\checkmark\!\!\!\! \ _{(J)}$ & $\checkmark\!\!\!\! \ _{(J)}$  \\
\hline
$\frn_{6}=(0,0,0,0,12,13)$ & $N_6$ & $-$ & $-$ & $\checkmark$  \\
\hline
$\frn_{8}=(0,0,0,0,0,12)$ & $N_8$ & $-$ & $\checkmark$ & $-$  \\
\hline
\end{tabular}
%}
\bigskip
\caption{The sign of $K_1$}
\label{table-sign}
\end{table}

Let us recall that the groups $N_{2}, N_3, N_4, N_5$ and $N_{6}$ admit left-invariant balanced metrics, while $N_3$ is the unique group admitting locally conformally K\"ahler metrics \cite{U}. We refer to \cite{LUV} for a classification of the complex structures satisfying $K_1<0$, $=0$ or $>0$.

\begin{remark}\rm
The nilpotent Lie group $N_3$ 
is given by the product of $\mathbb{R}$ with the 5-dimensional generalized Heisenberg group, while $N_5$ is the real Lie group underlying the Iwasawa manifold.
\end{remark}

We stress that, by an appropriate choice either of the Gauduchon connection $\nabla^{\tau}$ or of the slope parameter $\alpha'$, the sign of the constant $K_2=K_2(\omega_0,\alpha',\tau)$ in the model problem \eqref{model} can take any value. Therefore, the results presented in Section~\ref{qualit-model-prob} apply to every nilpotent Lie group in Table~\ref{table-sign}. In particular, we get

\begin{proposition}\label{prop_imm-anc}
%The Lie groups $N_2,\ldots,N_6$ and $N_8$ admit both immortal and ancient left-invariant solutions to the Anomaly flow~\eqref{AF-triv}. 
Any Lie group in Table \ref{table-sign} with $K_1\neq 0$ admits both immortal and ancient left-invariant solutions to the Anomaly flow \eqref{AF-triv}.
\end{proposition} 

This result also extends to nilmanifolds arising from the quotient of a Lie group $N_k$ by a co-compact lattice.
 
\subsection{Convergence of the nilmanifolds}\hfill\medskip

\noindent We are now in a position to prove our convergence result. Note that, a main ingredient in the proof of Theorem~\ref{thm_conv} will be given by the qualitative behaviour of the model problem studied in Section~\ref{qualit-model-prob}, together with Theorem~\ref{almost-diag-prop}.

\medskip
Let us recall that a family of compact metric spaces $(X_t,d_t)$ converges to a metric space $(\bar X, \bar d)$ in {\em Gromov-Hausdorff topology} as $t\to T$ , if for any increasing sequence $t_n\to T$ there exists a sequence of $\varepsilon_{t_n}$-approximations $\varphi_{t_n}:X_{t_n}\to \bar X$ satisfying $\varepsilon_{t_n}\to 0$. By definition, $\varphi:X\to \bar X$ is an $\varepsilon$-{\em approximation} if
$$
|d_t(x,x')-\bar d(\varphi(x),\varphi(x'))|<\varepsilon\,,\qquad \text{for any}\,\, x,x'\in X\,,
$$
and for all $y\in \bar X$ there exists $x\in X$ such that $\bar d(y,\varphi(x))<\varepsilon$ (see e.g. \cite{Rong}).

\begin{proof}[Proof of Theorem \ref{thm_conv}] Let $M=\Gamma\,\backslash\, G$ be a nilmanifold arising from our class of nilpotent Lie groups and let $\{\zeta^1,\zeta^2,\zeta^3\}$ be an invariant (1,0)-frame of $X=(M,J)$. By means of \eqref{J-nilp}, $M$ gives rise to a fibration over a real 4-dimensional tours $\pi:M\to \mathbb T^4$ with fibers spanned by the real and imaginary part of $\zeta^3$. On the other hand, by means of Theorem \ref{almost-diag-prop} and the results presented in Section~\ref{qualit-model-prob}, one gets that either $\{\zeta^1,\zeta^2,\zeta^3\}$ or $\{\zeta^3\}$ shrink to zero along $(1+t)^{-1}\omega_t$ as $t\to \infty$, depending on the signs of $K_1$ and $K_2$, and hence the claim follows. 
%Indeed, if we consider a not necessarily continuous function $f : \mathbb T^4 \to M$ satisfying $\pi \circ f = id$, then for any $\varepsilon > 0$ there exists $t_\ast(\varepsilon) > 0$ such that $(\pi, f)$ is a GH $\varepsilon$-approximation between  $(M, (1+t)^{-1}\omega_t)$  and $(\mathbb T^4)$ for any $t > t_\ast(\varepsilon)$.
\end{proof}

\section{Evolution of the holomorphic vector bundle}\label{sec-flow-example}

\noindent
In this section we study the Anomaly flow \eqref{AF-G} on a class of Lie groups belonging to \eqref{J-nilp}. Explicit computations will be performed on the nilpotent Lie group $N_3$. In particular, we will prove that under certain choices of initial metric and connections, the Anomaly flow converges to a (non-flat) solution of 
the Hull-Strominger-Ivanov system. \medskip

Let $G$ be a $6$-dimensional Lie group and let $J$ be a left-invariant non-parallelizable complex structure on $G$. Let us suppose that there exists a left-invariant $(1,0)$-coframe  $\{\zeta^1,\zeta^2,\zeta^3\}$ on $G$ satisfying the structure equations
\begin{equation}\label{J-nilp-lambda=0}
\begin{cases}
d \zeta^1=d\zeta^2=0,\cr d\zeta^3=\rho\, \zeta^{12} +\zeta^{1\bar{1}} + (x+i\,y)\,\zeta^{2\bar{2}}\,,
\end{cases}
\end{equation}
where $x,y\in \mathbb{R}$ and $\rho\in \{0,1\}$ (i.e. we are considering $\lambda=0$ in \eqref{J-nilp}). Let also the holomorphic vector bundle be $E:=T^{1,0}G$, and
$$
\Psi:=\zeta^1\wedge\zeta^2\wedge\zeta^3\,.
$$ 
Moreover, let the left-invariant Hermitian metrics $(\omega_0,H_0)$ be both diagonal, i.e.
$$
\omega_0  =  \frac{i}{2}\left( r_0^2\, \zeta^{1\bar{1}} + s_0^2\, \zeta^{2\bar{2}} + k_0^2\, \zeta^{3\bar{3}} \right)
$$
and
$$
H_0  =  \frac{i}{2}\left( \tilde r_0^2\, \zeta^{1\bar{1}} + \tilde s_0^2\, \zeta^{2\bar{2}} + \tilde k_0^2\, \zeta^{3\bar{3}}\right)\,.
$$
Then, our main result is the following 

\begin{theorem}\label{prop-diago}
The left-invariant metrics $\omega_t$ and $H_t$ solving the Anomaly flow \eqref{AF-G} remain diagonal along the flow, and the coefficients of $H_t$ evolve via 

\begin{equation}\label{evol-sec-eq}
\left\lbrace 
\begin{aligned}
\frac d{dt}\,\tilde r(t)^2 &= \frac{1}{3 c_1 c_2^2} \Big[ 2 \big( c_2 (\kappa+1)^2 - c_1\rho(\kappa-1)^2 \big) r(t)^2\tilde k(t)^2 \\[-4pt] 
& \hskip3.8cm
 - c_1c_2 (\kappa-1) (c_1 x+c_2)  \tilde r(t)^2 \Big] \frac{\tilde k(t)^2}{r(t)^4\tilde r(t)^2}\,,\\[6pt]
\frac d{dt}\,\tilde s(t)^2 &= \frac{1}{3 c_1^2 c_2}  
\Big[ 2 \big( c_1(\kappa+1)^2(x^2+y^2) - c_2\rho(\kappa-1)^2 \big) r(t)^2\tilde k(t)^2 \\[-4pt] 
& \hskip3.8cm
 - c_1^2 (\kappa-1) \big( c_1(x^2+y^2) + c_2 x \big) \tilde s(t)^2 \Big] \frac{\tilde k(t)^2}{r(t)^4\tilde s(t)^2}\,,\\[6pt]
\frac d{dt}\,\tilde k(t)^2 &= \frac{2}{3 c_1^2 c_2^2}  \Big[ \rho(\kappa-1)^2 \big(  c_1^2 \tilde s(t)^4 + c_2^2 \tilde r(t)^4 \big) 
\\[-4pt] 
& \hskip1.8cm
 - c_1c_2 (\kappa+1)^2 \big(  (x^2+y^2)\tilde r(t)^4 + \tilde s(t)^4 \big) \Big] \frac{\tilde k(t)^6}{r(t)^2\tilde r(t)^4\tilde s(t)^4}\,.
\end{aligned}
\right.
\end{equation}
\end{theorem}

To prove our statement, we need the following lemma.

\begin{lemma}\label{lemm_ass} Under the hypotheses of Theorem \ref{prop-diago}, 
$$
{\rm Tr}(A_0^\kappa \wedge A_0^\kappa)=C_0 \, \zeta^{12\bar{1}\bar{2}},
$$
where $C_0=C_0(\lambda,x,y;\omega_0,H_0;\kappa)$ is a constant depending both on the Hermitian structures and the connection $\nabla^\kappa$.
\end{lemma}

\begin{proof} The proof directly follows by Lemma~\ref{lemm_app} in Appendix~B. %%Appendix~\ref{apendiceB}.
\end{proof}

\begin{proof}[Proof of Theorem \ref{prop-diago}]
Let us focus on the evolution of $H_t$ via
\begin{equation}\label{sec-eq}
H_t^{-1} \partial_t\, H_t  =  \displaystyle\frac{\omega_t^2\wedge A_t^\kappa}{\omega_t^3}\,.
\end{equation}
We first show that there exists $\widetilde T>0$ such that $H_t$ holds diagonal for any $t\in[0,\widetilde T)$. To this end, it is enough to prove that $\omega_t^2\wedge(A_t^\kappa)_{\bar j}^i=0$ for any $i\not=j$ and $t=0$. Thus, let $H$ and $\omega$ be two left-invariant diagonal Hermitian metrics on $G$ given by
$$
H  =  \frac{i}{2}\left( \tilde r^2\, \zeta^{1\bar{1}} + \tilde s^2\, \zeta^{2\bar{2}} + \tilde k^2\, \zeta^{3\bar{3}}\right)\,,\qquad \tilde s^2,\tilde r^2,\tilde k^2>0\,,
$$
and
$$
\omega  =  \frac{i}{2}\left(  r^2 \,\zeta^{1\bar{1}} +  s^2\, \zeta^{2\bar{2}} + k^2\, \zeta^{3\bar{3}}\right)\,,\qquad s^2,r^2,k^2>0\,.
$$ 
If we consider $\{ e^1,\dots, e^6\}$ a left-invariant coframe on $G$ such that
$$
\delta_1\,\zeta^1=e^1\!+i\,e^2= e^1\!-\!i\,J e^1\,, \quad \delta_2\,\zeta^2=e^3\!+i\,e^4= e^3\!-\!i\,J e^3\,, \quad \delta_3\,\zeta^3=e^5\!+i\,e^6= e^5\!-\!i\,J e^5\,, 
$$
with $\delta_1=r$, $\delta_2=s$ and $\delta_3=k$, then we get 
%%%\begin{equation}\label{compl-curv-form}
$$
(A^\kappa)_{\bar j}^i= \frac{1}{\delta_i\delta_j}\,\left((A^\kappa)^{ e^i}_{ e^j}+i\, (A^\kappa)^{ e^i}_{J e^j}-i\, (A^\kappa)^{J e^i}_{ e^j}+(A^\kappa)^{J e^i}_{J e^j}\right)\,,
$$
%%%\end{equation}
where $(A^\kappa)^{ e^i}_{ e^j}$ are the curvature 2-forms of $\nabla^\kappa$ explicitly computed in Appendix~B 
%%%Appendix~\ref{apendiceB} 
(see the proof of Lemma~\ref{lemm_app}). 
Thus, the only non-zero entries in the right-hand side of \eqref{sec-eq} are given by
%$$
%\frac{F^2\wedge (A^\kappa)_{\bar j}^i}{F^3}
%$$ 
\begin{equation}\label{inst-diag}
\begin{aligned}
\frac{\omega^2\wedge (A^\kappa)^1_{\bar 1}}{\omega^3}&= \frac 1{12}\frac{\tilde k^2}{r^4s^2k^2\tilde r^4} \left[ r^2\tilde k^2 \left(  (\kappa+1)^2s^2-\rho(\kappa-1)^2r^2 \right) - 4 k^2 \tilde r^2 (\kappa-1) (x r^2+s^2)\right]\,,\\[5pt]
\frac{\omega^2\wedge (A^\kappa)^2_{\bar 2}}{\omega^3}&= \frac 1{12}\frac{\tilde k^2}{r^4s^2k^2\tilde s^4} 
\Big[ s^2\tilde k^2 \left(  (\kappa+1)^2(x^2+y^2)r^2-\rho(\kappa-1)^2s^2 \right) \\ 
& \hskip3.8cm
 - 4 k^2 \tilde s^2 (\kappa-1) \left( (x^2+y^2)r^2+ x\,s^2 \right) \Big]\,,\\[5pt]
\frac{\omega^2\wedge (A^\kappa)^3_{\bar 3}}{\omega^3}&= \frac 1{12}\frac{\tilde k^4}{r^4s^2k^2\tilde r^4\tilde s^4} \Big[ \rho(\kappa-1)^2 \left(  r^4 \tilde s^4 + s^4 \tilde r^4\right) - (\kappa+1)^2r^2s^2 \left(  (x^2+y^2)\tilde r^4 + \tilde s^4 \right) \Big]\,,
\end{aligned}
\end{equation}
and hence our claim follows, since $\omega_0$ and $H_0$ are both diagonal.

On the other hand, by means of Lemma \ref{lemm_ass} and Corollary \ref{cor-diag}, there also exists $\widehat T>0$ such that $\omega_t$ holds diagonal for any $t\in[0,\widehat T)$. Therefore, by the existence of $\widehat T>0$ and $\widetilde T>0$, it follows that $\omega_t$ and $H_t$ hold diagonal for any $t$ along the flow.
Finally, the evolution equations in \eqref{evol-sec-eq} are a direct consequence of \eqref{inst-diag}, taking into account that $s(t)^2=\frac{c_2}{c_1} r(t)^2$ and $k(t)^2=\frac{c_1c_2}{8}$ by Corollary~\ref{cor-diag}.
\end{proof}

\begin{remark} \rm Under the assumptions of Theorem \ref{prop-diago}, we have 
$$
{\rm Tr}(A_t^\kappa \wedge A_t^\kappa)=C_t \, \zeta^{12\bar{1}\bar{2}},
$$
where $C_t=C_t(\rho,x,y;\omega_t,H_t;\kappa)$ is a one-parameter function depending both on the Hermitian structures and the connection $\nabla^\kappa$.
\end{remark}

\begin{remark}\label{lambda=0}
{\rm 
Theorem~\ref{prop-diago} applies to the following Lie groups $N_k$ in Table~\ref{table-sign} and complex structures in \eqref{J-nilp-lambda=0}:  
{\bf\emph{(1)}} $\rho=0$, $y=1$ and $x\in \mathbb{R}$, the Lie group is $N_2$;
{\bf\emph{(2)}} $\rho=y=0$ and $x=\pm 1$, the Lie group is $N_3$;
{\bf\emph{(3)}} $\rho=1$, $y\geq 0$ and $1+4x>4y^2$, the Lie group is $N_5$;
{\bf\emph{(4)}} $\rho=x=y=0$, the Lie group is $N_8$.
By \cite{COUV}, this is a classification of all the complex structures in \eqref{J-nilp-lambda=0}.
Regarding the existence of balanced Hermitian metrics, the list reduces to:

\smallskip

$\bullet$  $\rho=y=0$, $x=-1$, the Lie group is $N_3$;

\smallskip

$\bullet$  $\rho=1$, $y=0$ and $x \in (-1/4,0)$, the Lie group is $N_5$.
}
\end{remark}

Our next result shows that if the initial metric $\omega_0$ is balanced, then there always exists a connection $\nabla^\kappa$ such that \eqref{evol-sec-eq} only admits constant solutions.

\begin{proposition}\label{prop-diago-always-sol}
Under the hypotheses of Theorem \ref{prop-diago}, if the initial metric $\omega_0$ is balanced, then there exists a Gauduchon connection $\nabla^\kappa$ for which the right-hand side of the system \eqref{evol-sec-eq} identically vanishes, and the only admissible solutions to the Anomaly flow \eqref{AF-G} are those with constant $H_t$, i.e $H_t \equiv H_0$.
\end{proposition}

\begin{proof}
As we already showed in the proof of Theorem \ref{bal-cond}, a diagonal metric $\omega_0$ is balanced if and only if $c_1 (x+i\, y)+c_2=0$, which is equivalent to require 
\begin{equation}\label{balanced_prop}
x=-\frac{c_2}{c_1}\qquad \text{and}\qquad  y=0\,,
\end{equation}
with $c_1=\frac{\sqrt{8}\,r_0k_0}{s_0}>0$ and $c_2=\frac{\sqrt{8}\,s_0k_0}{r_0}>0$ by Corollary \ref{cor-diag}. Therefore, by means of \eqref{balanced_prop}, the system \eqref{evol-sec-eq} can be written as
$$
%%%\begin{equation}\label{evol-sec-eq-bis}
\left\lbrace 
\begin{aligned} 
\frac d{dt}\,\tilde r(t)^2 &= \frac{2}{3 c_1 c_2^2} \Big( c_2 (\kappa+1)^2 - c_1\rho(\kappa-1)^2 \Big) \frac{\tilde k(t)^4}{r(t)^2\tilde r(t)^2}\,,\\[6pt]
\frac d{dt}\,\tilde s(t)^2 &= \frac{2}{3 c_1^3}  
\Big( c_2(\kappa+1)^2 - c_1\rho(\kappa-1)^2 \Big) \frac{\tilde k(t)^4}{r(t)^2\tilde s(t)^2}\,,\\[6pt]
\frac d{dt}\,\tilde k(t)^2 &= \frac{2}{3 c_1^3 c_2^2}  \Big( c_1\rho(\kappa-1)^2  - c_2 (\kappa+1)^2  \Big)
\frac{\big(  c_1^2 \tilde s(t)^4 + c_2^2 \tilde r(t)^4 \big) \, \tilde k(t)^6}{r(t)^2\tilde r(t)^4\tilde s(t)^4}\,.
\end{aligned} 
\right.
%%%\end{equation}
$$
Then, the right-hand side of the system identically vanishes if and only if 
\begin{equation}\label{polynomial}
c_2 (\kappa+1)^2 - c_1 \rho\, (\kappa-1)^2= (c_2-\rho\,c_1) \kappa^2 + 2(c_2+\rho\,c_1) \kappa + c_2-\rho\,c_1=0\,.
\end{equation}
Finally, since the discriminant of this quadratic polynomial is given by 
$$
\Delta=32 \rho c_1c_2 \geq 0\,,
$$ 
the claim follows.
\end{proof} 

Remarkably, by the proof of Proposition \ref{prop-diago-always-sol}, we can distinguish the following two remarkable cases:
\begin{enumerate}\setlength\itemsep{0.5em}
\item[$\bullet$]
If $c_2=\rho\,c_1$, then the only solution to the polynomial \eqref{polynomial} is $\kappa=0$. Therefore, $\nabla^\kappa$ is given by the \emph{Lichnerowicz connection} $\nabla^0$.
\item[$\bullet$] If $c_2\ne \rho\,c_1$, then the solutions to \eqref{polynomial} are either the Bismut connection ($\kappa=-1$) when $\rho=0$,  or the Gauduchon connections $\nabla^{\kappa^{\pm}}$ corresponding to the values 
$\kappa^{\pm}=\frac{c_1 + c_2 \pm 2\sqrt{c_1c_2}}{c_1 - c_2}$ 
when $\rho=1$.
\end{enumerate}

\begin{remark}\rm Given a solution $H_t$ to \eqref{evol-sec-eq}, it may happen that its Gauduchon connection  $\nabla^{\kappa}_{t}$ does not satisfy the condition $({A^{\kappa}_{t}})^{2,0}=({A^{\kappa}_{t}})^{0,2}=0$. For instance, let us consider the Lie group arising from \eqref{J-nilp-lambda=0} when $\rho=1$, $x=-\frac{1}{8}$ and $y=0$, which corresponds to $N_5$ (see Remark~\ref{lambda=0}). Then, the diagonal metric 
$\omega_0= \frac{i}{2}\left( \zeta^{1\bar{1}} + \frac{1}{8}\, \zeta^{2\bar{2}} + \zeta^{3\bar{3}} \right)$ on $N_5$ is balanced and, by means of Proposition~\ref{prop-diago-always-sol}, one gets that for $\kappa^{\pm}=\frac{9\pm 4\sqrt{2}}{7}$ the system \eqref{evol-sec-eq} is solved by the constant metric $H_t\equiv \frac{i}{2}\left( \zeta^{1\bar{1}} + \zeta^{2\bar{2}} + \zeta^{3\bar{3}} \right)$. Nonetheless, the Gauduchon connections $\nabla^{\kappa^{\pm}}$ do not satisfy the condition $({A^{\kappa^{\pm}}_{t}})^{2,0}=({A^{\kappa^{\pm}}_{t}})^{0,2}=0$. Indeed, by Appendix~B 
%%%Appendix~\ref{apendiceB}
(see the proof of Lemma~\ref{lemm_app}) one gets that $(A^{\kappa^{\pm}})^{ 1}_{ 2}$ has non-zero component in 
$e^{14}+e^{23}=\frac{-i}{4\sqrt{2}}(\zeta^{12}-\zeta^{\bar{1}\bar{2}})$, and hence its $(2,0)$ and $(0,2)$ components do not identically vanish.
\end{remark}
%
%\medskip
%
%Therefore, an stationary point $T$ of the Anomaly flow \eqref{AF-G} may not give rise to a solution of the Hull-Strominger system. Indeed, 
%%%%as the previous example shows, 
%one needs that the additional condition $({A^{\kappa}_{T}})^{2,0}=0=({A^{\kappa}_{T}})^{0,2}$ be satisfied for the connection $\nabla^{\kappa}_T$. 

In the following section we prove that, on the Lie group~$N_3$, solutions to the Hull-Strominger-Ivanov system can obtained as stationary points to the Anomaly flow.

\subsection{Anomaly flow on $N_3$ and solutions to the Hull-Strominger-Ivanov system}\hfill\smallskip

\noindent Let us consider the simply-connected nilpotent Lie group $N_3$, which admits a left-invariant $(1,0)$-coframe $\{\zeta^1,\zeta^2,\zeta^3\}$ satisfying the structure equations
\begin{equation}\label{N_group}
\begin{cases}
d \zeta^1=d\zeta^2=0,\cr d\zeta^3=\zeta^{1\bar{1}} -\zeta^{2\bar{2}}\,.
\end{cases}
\end{equation}

Next we study the Anomaly flow \eqref{AF-G} on $N_3$ for $\nabla^\kappa$ being the Chern connection (i.e. $\kappa=1$) and the Strominger-Bismut connection (i.e. $\kappa=-1$).

\subsubsection*{$\bullet$ The Chern connection on $T^{1,0}N_3$}\hfill\smallskip

\noindent We start investigating the setting of Theorem \ref{prop-diago} in the special case when $\kappa=1$, i.e. $\nabla_t^\kappa$ is the Chern connection on $(T^{1,0}N_3,H_t)$.

\begin{theorem}\label{h3-Chern-instanton} If $\kappa=1$, then the coefficients of $\omega_t$ and $H_t$ evolve via the ODEs system
\begin{equation}\label{flow-h3-instanton-Chern}
\left\{
\begin{aligned}
\frac d{dt}\, r(t)^2 & =  \frac{c_1^2c_2}{2^5} +\alpha'(1-\tau)(\tau^2-2\tau+5)\, \frac{c_1^3(c_1^2+c_2^2)}{2^{11}\,r(t)^4}\,,\\
\frac d{dt}\,\tilde r(t)^2 &= \frac{8}{3 c_1 c_2}  \frac{\tilde k(t)^4}{r(t)^2\tilde r(t)^2}\,,\\[5pt]
\frac d{dt}\,\tilde s(t)^2 &= \frac{8}{3 c_1 c_2}  
\frac{\tilde k(t)^4}{r(t)^2\tilde s(t)^2}\,,\\[5pt]
\frac d{dt}\,\tilde k(t)^2 &= -\frac{8}{3 c_1 c_2} 
 \Big( \tilde r(t)^4 + \tilde s(t)^4 \Big) \frac{\tilde k(t)^6}{r(t)^2\tilde r(t)^4\tilde s(t)^4}\,.
\end{aligned}
\right.
\end{equation}
Moreover, if $\omega_0$ and $H_0$ are both balanced, then $H_t$ evolves as 
$$
H_t  =  \frac{i}{2} \tilde r(t)^2 \zeta^{1\bar{1}} + \frac{i}{2}\tilde r(t)^2 \zeta^{2\bar{2}} + \frac{i}{2} \frac{\tilde r_0^4 \tilde k_0^2}{\tilde r(t)^4} \, \zeta^{3\bar{3}}\,,
$$
where the function $\tilde r(t)^2$ satisfies 
\begin{equation}\label{tilde-r}
\frac{d}{dt}\, \tilde r(t)^2=\frac 8{3c_1^2}\frac{\tilde r(0)^8 \tilde k(0)^4}{r(t)^2\tilde r(t)^{10}}
\end{equation}
%$$
%\frac d{dt}\!\Big(\tilde{r}(t)^{12}\Big)= 16 \,\frac{\tilde r_0^8 \tilde k_0^4}{c_1^2}\,\frac{1}{r(t)^2}.
%$$
In particular, if $\tau\not=1$ (i.e.  $\nabla_t^\tau$ is different from the Chern connection), then there exists a convenient choice of $\alpha'$ such that 
the solution to the system is given by $\omega_t\equiv\omega_0$ and 
$\tilde r(t) = \displaystyle{\sqrt[12]{A\, t +B}}$, with
$A=\frac{16\,\tilde r_0^8\, \tilde k_0^4}{c_1^2\,r_0^2}$ and $B=\tilde r_0^{12}$. 
\end{theorem}

\begin{proof} 
By means of Proposition~\ref{AF-evo-F}, the first equation of the Anomaly flow \eqref{AF-G} reduces to
\begin{equation}\label{first-eq-syst}
\frac d{dt}\, r(t)^2= \frac{c_1}{4}\, K(t,\alpha',\tau)\,,
\end{equation}
where $K(t,\alpha',\tau)$ is given by
$$
K(t,\alpha',\tau)\,\zeta^{12\bar 1\bar 2} = i \partial\overline\partial \omega_t - \frac{\alpha'}{4}\left( {\rm Tr}(Rm^\tau_t\wedge Rm^\tau_t) - {\rm Tr}(A^1_t\wedge A^1_t)\right)\,.
$$
By Corollary~\ref{cor-diag} and Proposition~\ref{traza-tau}, a direct computation yields that
$$
\begin{aligned}
&i\partial\overline\partial \omega_t = \frac{c_1c_2}{2^3} \, \zeta^{12\bar{1}\bar{2}}\,,\\
&{\rm Tr}(Rm_t^{\tau}\wedge Rm_t^{\tau}) = (\tau-1)(\tau^2-2\tau+5)\, \frac{c_1^2+c_2^2}{2^7}\, \frac{c_1^2}{r(t)^4} \, \zeta^{12\bar{1}\bar{2}}\,,
\end{aligned}
$$
while, by means of 
Lemma~\ref{lemm_app}, for $\kappa=1$ we have ${\rm Tr}(A^1_t\wedge A^1_t) =0$. Therefore, by using \eqref{first-eq-syst} and \eqref{evol-sec-eq} for $\kappa=1$, $\rho=0=y$ and $x=-1$, one gets the ODEs system \eqref{flow-h3-instanton-Chern}.\smallskip

Now, let $\omega_0$ and $H_0$ be both balanced. By means of \eqref{Bnilp-condition} and \eqref{ecu-1}--\eqref{ecu-5}, the balanced condition implies that
\begin{equation*}\label{bal-ex}
c_2=c_1 \qquad\text{and}\qquad \tilde s_0^2=\tilde r_0^2\,.
\end{equation*}
The latter equality, together with the fact that the functions $\tilde r(t)^2$ and $\tilde s(t)^2$ satisfy similar equations in \eqref{flow-h3-instanton-Chern}, 
leads to $\tilde s(t)^2=\tilde r(t)^2$. Thus, 
the ODEs system \eqref{flow-h3-instanton-Chern} reduces to
\begin{equation}\label{k1}
\left\{
\begin{aligned}
\frac d{dt}\, r(t)^2 & =   \frac{c_1^3}{2^5} +\alpha'(1-\tau)(\tau^2-2\tau+5)\, \frac{c_1^5}{2^{10}\,r(t)^4}\,,\\
\frac d{dt}\,\tilde r(t)^2 &= \frac{8}{3 c_1^2}  \frac{\tilde k(t)^4}{r(t)^2\tilde r(t)^2}\,,\\[5pt]
\frac d{dt}\,\tilde k(t)^2 &= -\frac{16}{3 c_1^2} \frac{\tilde k(t)^6}{r(t)^2\tilde r(t)^4}\,.
\end{aligned}
\right.
\end{equation}
%%%On the other hand, since $\tilde r_0^2=\tilde s_0^2$, we get that $$\tilde r(t)^2=\tilde s(t)^2\,,$$ and hence \eqref{flow-h3-instanton-Chern2} follows.
Therefore, by considering the quotient of $\frac{d}{dt} \tilde r(t)^2$ with $\frac{d}{dt} \tilde k(t)^2$, we get
$$
\int{\frac1{\tilde r(t)^2}\,{{\rm d}\tilde r(t)^2}}=-\frac 12 \int{\frac1 {\tilde k(t)^2}\,{{\rm d}\tilde k(t)^2}}\,.
$$
This in turn implies 
$$
\tilde k(t)=\frac{\tilde r_0^2 \tilde k_0}{\tilde r(t)^2}\,,
$$ 
and hence \eqref{tilde-r} follows. 
%Now, the second equation in \eqref{k1} writes as
%$$
%\frac d{dt}\!\Big(\tilde{r}(t)^{12}\Big)= 6\,\tilde r(t)^{10}\, \frac d{dt}\!\Big(\tilde{r}(t)^2\Big) =16 \,\frac{\tilde r_0^8 \tilde k_0^4}{c_1^2}\,\frac{1}{r(t)^2}.
%$$

Finally, for any value of $r_0^2$ and $\tau\neq 1$, there exists a convenient value of $\alpha'$ making the right-hand side of the first equation in \eqref{k1} equal to zero.
In this case we can explicitly solve the system with 
$$
\tilde r(t) = \displaystyle{\sqrt[12]{A\, t +B}},
$$
where $A=\frac{16\,\tilde r_0^8\, \tilde k_0^4}{c_1^2\,r_0^2}$ and $B=\tilde r_0^{12}$.
\end{proof}

We stress that the explicit solutions found in Theorem~\ref{h3-Chern-instanton} are not stationary solutions to the flow, and hence they do not solve the Hull-Strominger system. In the next subsection, we will construct stationary solutions assuming $\nabla_t^\kappa$ to be the Strominger-Bismut connection.

\subsubsection*{$\bullet$ The Bismut connection on $T^{1,0}N_3$}\hfill\smallskip

\noindent Here we consider the setting of Theorem \ref{prop-diago} in the special case when $\kappa=-1$, i.e. $\nabla_t^\kappa$ is the Strominger-Bismut connection on $(T^{1,0}N_3,H_t)$.

\begin{theorem}\label{h3-Bismut-instanton}
If $\kappa=-1$, then the coefficients of $\omega_t$ and $H_t$ evolve via the ODEs system
\begin{equation}\label{flow-h3-instanton-Bismut}
\left\{
\begin{aligned}
%\frac d{dt}\, r(t)^2 & =  \frac{c_1^2c_2}{2^5} +\alpha'(1-\tau)(\tau^2-2\tau+5)\, \frac{c_1^3(c_1^2+c_2^2)}{2^{11}\,r^4(t)}- \alpha'\,\frac{c_1^3(c_1^2+c_2^2)}{2^7\,r^4(t)}\,,\\
\frac d{dt}\, r(t)^2 & =  \frac{c_1^2c_2}{2^5} +\alpha'(1-\tau)(\tau^2-2\tau+5)\, \frac{c_1^3(c_1^2+c_2^2)}{8^4\,r(t)^4}-\alpha'\,\frac {c_1}{2}\, \frac {\tilde k(t)^4(\tilde r(t)^4+\tilde s(t)^4)}{\tilde r(t)^4\tilde s(t)^4}\,,\\
\frac d{dt}\,\tilde r(t)^2 &= \frac{2}{3 c_2}\left(c_2-c_1\right)\frac{\tilde k(t)^2}{r(t)^4}\,,\\
\frac d{dt}\,\tilde s(t)^2 &= \frac{2}{3 c_2}\left( c_1-c_2 \right)\frac{\tilde k(t)^2}{r(t)^4}\,,\\
\frac d{dt}\,\tilde k(t)^2 &= 0\,.
\end{aligned}
\right.
\end{equation}
If the initial metric $\omega_0$ is balanced, then $H_t\equiv H_0$ is constant, the Strominger-Bismut connection $\nabla^{-1}$ of $H_0$ is a (non-flat) instanton with respect to $\omega_t$, and the Anomaly flow reduces to the ODE
\begin{equation}\label{ODE-h3-instanton-Bismut}
\frac d{dt}\, r(t)^2= {K}_1 + \frac{K_2}{r(t)^4}\,,
\end{equation}
where $K_1=K_1(\omega_0,\alpha',H_0)$ and $K_2=K_2(\omega_0,\alpha',\tau)$ 
are given by
\begin{equation}\label{K_1-K_2-instanton}
K_1:=\frac{c_1^3}{2^5}-\alpha'\,\frac {c_1}{2}\, \frac {\tilde k_0^4(\tilde r_0^4+\tilde s_0^4)}{\tilde r_0^4\tilde s_0^4}\,,\ \qquad\  K_2:=\alpha'(1-\tau)(\tau^2-2\tau+5)\, \frac{c_1^5}{2^{10}}\,.
\end{equation}
Therefore, starting from any balanced initial metric $\omega_0$ on $N_3$, we have the following:
\begin{itemize}
\setlength\itemsep{0.5em}
\item[{\rm (i)}]
Given $\alpha'\not=0$ and $\tau\in \RR$, the metric $H_0$ can be conveniently chosen in order to obtain  $K_1<0$, $=0$ or $>0$ in~$\eqref{ODE-h3-instanton-Bismut}$, and so there always exists  a stationary point to the Anomaly flow which solves the Hull-Strominger system. 
\item[{\rm (ii)}]
Furthermore, if $\alpha'\not=0$ and $\tau=-1$ (i.e. $\nabla_t^\tau$ is the Strominger-Bismut connection of $\omega_t$),
then $\nabla_t^{-1}$ is an instanton with respect to $\omega_t$, and hence there exists a stationary point to the Anomaly flow which solves the Hull-Strominger-Ivanov system.
\end{itemize}
\end{theorem}

\begin{proof} The first part of the statement follows the same argument of Theorem \ref{h3-Chern-instanton}. We just mention that, by means of Lemma~\ref{lemm_app} for $\kappa=-1$, $\rho=0=y$ and $x=-1$, we have 
\begin{equation}\label{instanton-non-flat}
{\rm Tr}(A^{-1}_t\wedge A^{-1}_t) =-8\, \frac {\tilde r(t)^4+\tilde s(t)^4}{\tilde r(t)^4\tilde s(t)^4} \, \tilde k(t)^4\, \zeta^{12\bar{1}\bar{2}}\,.
\end{equation}
Hence, the ODEs system~\eqref{flow-h3-instanton-Bismut} is obtained starting from \eqref{evol-sec-eq}.

Now, let us assume $\omega_0$ balanced. By means of \eqref{Bnilp-condition} and \eqref{ecu-1}--\eqref{ecu-5}, we have 
$$
c_1=c_2\,,
$$ 
and hence the ODEs system \eqref{flow-h3-instanton-Bismut} reduces to 
$\tilde r(t)$, $\tilde s(t)$, $\tilde k(t)$ constant (i.e. $H_t\equiv H_0$), and
%the ODE equation \eqref{ODE-h3-instanton-Bismut}, that is
$$
\frac d{dt}\, r(t)^2= {K}_1 + \frac{K_2}{r(t)^4}\,,
$$
with $K_1$ and $K_2$ as given in \eqref{K_1-K_2-instanton}. 
Therefore, we get that $\omega_t^2\wedge A^{-1}=0$ for any $t\in(T_-,T_+)$ and, 
by means of the curvature forms given in the proof of Lemma~\ref{lemm_app}, a direct computation yields that the curvature of the Strominger-Bismut connection satisfies
\begin{equation}\label{instanto-bismut}
 ({A^{-1}})^{2,0}=({A^{-1}})^{0,2}=0\,.
\end{equation}
Hence, $\nabla^{-1}$ is an instanton with respect to $\omega_t$ for any $t\in(T_-,T_+)$. It is non-flat because ${{\rm Tr}(A^{-1}\wedge A^{-1}) \not=0}$ by \eqref{instanton-non-flat}. \medskip

Finally, the last two claims are consequences of \eqref{K_1-K_2-instanton}, \eqref{instanto-bismut} and similar arguments to those in Section \ref{sec-model}. In greater detail,
given $\alpha'\not=0$ and $\tau\in\RR$, we can choose a metric $H_0$ such that $K_1$ has opposite sign to that of $K_2$ (notice that if $K_2=0$ then we can take $H_0$ so that $K_1$ also vanishes). Now, (i) follows from a qualitative analysis similar 
to that given in Section \ref{sec-model}. For the proof of (ii), it only remains to prove that the Strominger-Bismut connection $\nabla_t^{-1}$ of the metric $\omega_t$ is an instanton with respect to $\omega_t$. This follows from Appendix~A 
for $\tau=-1$, $\rho=\lambda=y=0$, $x=-1$ and $u=0$ (because the metric $\omega_t$ remains diagonal). Indeed, in this case it is direct to check that $({Rm_t^{-1}})^{2,0}=({Rm_t^{-1}})^{0,2}=0$, so there exists a stationary point to the Anomaly flow which solves the system given by \eqref{HS} and \eqref{extra-HS}, i.e. the Hull-Strominger-Ivanov system.
\end{proof}

By Theorem \ref{h3-Bismut-instanton} (ii), when both $\nabla_t^\tau$ and $\nabla_t^\kappa$ are Strominger-Bismut connections
and the initial metric $\omega_0$ is balanced, the Anomaly flow \eqref{AF-G} always converges to a solution of the Hull-Strominger-Ivanov system. 
We note that explicit solutions of this kind were previously found in \cite[Theorem 5.1]{FIUV} and in \cite[Theorem 3.3]{OUV} by means of  other methods.

%%%
%%%\appendix
\section{Appendix A}\label{apendiceA}

\noindent In this Appendix we provide the curvature forms for a connection $\nabla^\tau$ in the Gauduchon family given any left-invariant metric $\omega$. We will denote the curvature by $\Omega^{\tau}$ instead of $Rm^{\tau}$.

For the computation of the curvature forms we use \eqref{curvature} with respect to the
adapted basis $\{e^l\}_{l=1}^6$ found in Proposition~\ref{adapted-ecus}, together with 
the connection 1-forms $(\sigma^{\tau})^i_j$ obtained in 
Proposition~\ref{traza-tau}.

We first notice that the 2-forms $(\Omega^{\tau})^i_j$ satisfy the following relations:
$$
\begin{aligned}
&(\Omega^{\tau})^2_3 = - (\Omega^{\tau})^1_4\,, \quad (\Omega^{\tau})^2_4 = (\Omega^{\tau})^1_3\,,\ \   &(\Omega^{\tau})^2_5 = - (\Omega^{\tau})^1_6\,,\quad &(\Omega^{\tau})^2_6 = (\Omega^{\tau})^1_5\,,\\
&(\Omega^{\tau})^4_5 = - (\Omega^{\tau})^3_6\,, \quad (\Omega^{\tau})^4_6 = (\Omega^{\tau})^3_5\,.
\end{aligned}
$$

Next, we give the explicit expression of the 2-forms $\tfrac{2r_e^4 \Delta_e^2}{k_e^2} (\Omega^{\tau})^i_j$, where 
$\Delta_e:=\sqrt{r_e^2s_e^2-|u_{e}|^2}$, for $(i,j)=\{(1,2), (1,3), (1,4), (1,5), (1,6), (3,4), (3,5), (3,6), (5,6)\}$:

$$
\begin{aligned}
%%% 1,2
\tfrac{2r_e^4 \Delta_e^2}{k_e^2} (\Omega^{\tau})^1_2 & =  
- {\scriptstyle(\tau^2\!-\!2\tau+5) \Delta_e^2}\, e^{12} 
+{\scriptstyle(\tau^2\!-\!2\tau+5)u_{e1} \Delta_e}\, (e^{13}+e^{24})\\
&
+ {\scriptstyle\big(  (\tau^2\!-\!2\tau+5)u_{e2} - \big(\rho(\tau\!-\!1)(\tau+3)+\lambda(\tau^2+3)\big) r_e^2\big) \Delta_e}\, e^{14}\\
&
- {\scriptstyle\big(  (\tau^2\!-\!2\tau+5)u_{e2} + \big( \rho(\tau\!-\!1)(\tau+3)-\lambda(\tau^2+3) \big) r_e^2\big) \Delta_e}\, e^{23}\\
&
- {\scriptstyle\big( 
 (\tau^2\!-\!2\tau+5) |u_{e}|^2  - 2 \lambda (\tau^2+3) u_{e2} r_e^2 
- (\rho (\tau\!-\!1)^2 - \lambda^2 (\tau+1)^2 + 4 x (\tau\!-\!1) ) r_e^4
\big)}\, e^{34}\\
&
- {\scriptstyle\lambda  (\tau\!-\!1)^2 \big(\lambda r_e^2-2 u_{e2}\big) r_e^2}\, e^{56}
\,,
\end{aligned}
$$

$$
\begin{aligned}
%%% 1,3
\tfrac{2r_e^4 \Delta_e^2}{k_e^2} (\Omega^{\tau})^1_3 & = 
\ {\scriptstyle(\tau^2\!-\!2\tau+5) u_{e1} \Delta_e}\, e^{12} \\
&
-{\scriptstyle\big[
(\tau^2\!-\!2\tau+5) u_{e1}^2 
+ \frac12 \big( \rho (\tau\!-\!1)^2 
-2 \lambda^2 (\tau\!-\!1)- \rho \lambda (\tau\!-\!1) (\tau+3) + x (\tau+1)^2 
\big) r_e^4
\big] }\, e^{13}\\
&
-{\scriptstyle\big[
(\tau^2\!-\!2\tau+5) u_{e1} u_{e2} 
- \big( \rho(\tau\!-\!1) (\tau+3)+\lambda(\tau^2+3) \big) u_{e1} r_e^2 
+ \frac{y}{2} (\tau+1)^2 r_e^4
\big]} \, e^{14}\\
&
+{\scriptstyle\big[
(\tau^2\!-\!2\tau+5) u_{e1} u_{e2} 
+ \big( \rho(\tau\!-\!1) (\tau+3)-\lambda(\tau^2+3) \big) u_{e1} r_e^2 
+ \frac{y}{2} (\tau+1)^2 r_e^4
\big] }\, e^{23}\\
&
-{\scriptstyle\big[
(\tau^2\!-\!2\tau+5) u_{e1}^2 
+ \frac12 \big( \rho (\tau\!-\!1)^2 
-2 \lambda^2 (\tau\!-\!1) + \rho \lambda (\tau\!-\!1) (\tau+3) + x (\tau+1)^2 
\big) r_e^4
\big]} \, e^{24}\\
&
+{\scriptstyle\frac{1}{\Delta_e} \big[(\tau^2\!-\!2\tau+5) |u_e|^2 u_{e1} 
- 2 \lambda (\tau^2+3) u_{e1} u_{e2} r_e^2} \\
&\quad\quad\quad
{\scriptstyle+\big( \lambda^2 (\tau^2+3) u_{e1} + x (\tau^2\!-\!2\tau+5) u_{e1} + y (\tau+1)^2 u_{e2} \big) r_e^4 
- \lambda y (\tau^2+3)  r_e^6 \big]} \, e^{34}\\
&
+ {\scriptstyle\frac{1}{\Delta_e}  \,
 (\tau\!-\!1)^2  
\big(\lambda r_e^2-2 u_{e2} \big)  \big(\lambda u_{e1} - y r_e^2 \big)     r_e^2
}\, e^{56}
\,,
\end{aligned}
$$

$$
\begin{aligned}
%%% 1,4
\tfrac{2r_e^4 \Delta_e^2}{k_e^2} (\Omega^{\tau})^1_4  & = 
{\scriptstyle\big(  (\tau^2\!-\!2\tau+5)u_{e2} + 2 \lambda (\tau\!-\!1) r_e^2\big) \Delta_e}\, e^{12}\\
&
-{\scriptstyle\big[
(\tau^2\!-\!2\tau+5) u_{e1} u_{e2} 
+2\lambda(\tau\!-\!1) u_{e1} r_e^2 
- \frac{y}{2} (\tau+1)^2 r_e^4
\big]} \, (e^{13}+e^{24})\\
&
-{\scriptstyle\big[
(\tau^2\!-\!2\tau+5) u_{e2}^2 
- \big( \rho(\tau\!-\!1) (\tau+3)+\lambda(\tau^2\!-\!2\tau+5) \big) u_{e2} r_e^2 } \\
&\quad\quad\quad
{\scriptstyle
+ \frac12 \big( \rho (\tau\!-\!1)^2 
-2 \lambda^2 (\tau\!-\!1) + \rho \lambda (\tau\!-\!1) (\tau+3) + x (\tau+1)^2 
\big) r_e^4
\big]} \, e^{14}\\
&+{\scriptstyle\big[
(\tau^2\!-\!2\tau+5) u_{e2}^2 
+ \big( \rho(\tau\!-\!1) (\tau+3)-\lambda(\tau^2\!-\!2\tau+5) \big) u_{e2} r_e^2  } \\
&\quad\quad\quad
{\scriptstyle
+ \frac12 \big( \rho (\tau\!-\!1)^2 
-2 \lambda^2 (\tau\!-\!1) - \rho \lambda (\tau\!-\!1) (\tau+3) + x (\tau+1)^2 
\big) r_e^4
\big] }\, e^{23}\\
& +{\scriptstyle\frac{1}{\Delta_e} \big[(\tau^2\!-\!2\tau+5) |u_e|^2 u_{e2} 
+ 2 \lambda \big( (\tau\!-\!1) u_{e1}^2 - (\tau^2\!-\!\tau+4) u_{e2}^2 \big) r_e^2} \\
&\quad\quad\quad
{\scriptstyle+\big( \lambda^2 (\tau^2+3) u_{e2} + x (\tau^2\!-\!2\tau+5) u_{e2} - y (\tau+1)^2 u_{e1} \big) r_e^4 
- \lambda x (\tau^2+3)  r_e^6 \big]} \, e^{34}\\
&
- {\scriptstyle\frac{1}{\Delta_e}  \,
 (\tau\!-\!1)^2  
\big(2\lambda u_{e2}^2 - \big(\lambda s_e^2 +\lambda^2 u_{e2} - 2 y u_{e1} \big)  r_e^2 +\lambda x r_e^4
\big) r_e^2}\, e^{56}
\,,
\end{aligned}
$$

$$
\begin{aligned}
%%% 1,5 
\tfrac{2r_e^4 \Delta_e^2}{k_e^2} (\Omega^{\tau})^1_5  & = 
-{\scriptstyle\frac{\lambda}{2} (\tau\!-\!1) \big( (\tau+1) u_{e2} -\rho (\tau\!-\!1) r_e^2 \big) r_e^2} \, e^{15} 
+{\scriptstyle (\tau\!-\!1) \big( \rho (\tau\!-\!1) - \frac{\lambda}{2} (\tau+1) \big) u_{e1} r_e^2 }\, e^{16} \\
&
{\scriptstyle-\frac{\lambda}{2} (\tau\!-\!1) (\tau+1) u_{e1} r_e^2 }\, e^{25} \\
&
-{\scriptstyle\frac{1}{2} (\tau\!-\!1) \big( 2 (\tau+1) s_e^2 + 2\rho (\tau\!-\!1) u_{e2} -\lambda (\tau+1) u_{e2} - \rho\lambda(\tau\!-\!1) r_e^2 \big) r_e^2 }\, e^{26} \\
&
+{\scriptstyle\frac{\lambda}{2\Delta_e} (\tau\!-\!1) (\tau + 1) \big( |u_e|^2 -\lambda u_{e2} r_e^2 + x r_e^4 \big) r_e^2} \, e^{35}  \\
&
+{\scriptstyle\frac{1}{2\Delta_e} (\tau\!-\!1) (\tau + 1) \big( 2u_{e1} s_e^2 -(\lambda^2 u_{e1} -2xu_{e1}+2 y u_{e2}) r_e^2 + \lambda y r_e^4 \big) r_e^2} \, e^{36}  \\
&
-{\scriptstyle\frac{\lambda}{2\Delta_e} (\tau\!-\!1) (\tau + 1) \big( \lambda u_{e1} - y r_e^2  \big) r_e^4 }\, e^{45} \\
&
+{\scriptstyle\frac{1}{2\Delta_e} (\tau\!-\!1)  \big[ 2\rho(\tau\!-\!1)|u_e|^2 + (\tau+1) (2s_e^2 u_{e2} - \lambda |u_e|^2)  } \\
&\quad\quad\quad\quad\quad
{\scriptstyle
- \big( 2 \rho (\tau\!-\!1) s_e^2 + (\tau+1) (2 \lambda s_e^2 \!-\! \lambda^2 u_{e2} \!-\! 2 x u_{e2} \!-\! 2 y u_{e1}) \big) r_e^2     
-  \lambda x (\tau+1) r_e^4 \big] r_e^2} \, e^{46} 
\,,
\end{aligned}
$$

$$
\begin{aligned}
%%% 1,6 ...........
\tfrac{2r_e^4 \Delta_e^2}{k_e^2} (\Omega^{\tau})^1_6  & = 
-{\scriptstyle\frac{\lambda}{2} (\tau\!-\!1) (\tau+1) u_{e1} r_e^2 }\, e^{15} 
-{\scriptstyle\frac{1}{2} (\tau\!-\!1) \big( 2 (\tau+1) s_e^2 - 2\rho (\tau\!-\!1) u_{e2} -\lambda (\tau+1) u_{e2} + \rho\lambda(\tau\!-\!1) r_e^2 \big) r_e^2} \, e^{16} \\
&
+{\scriptstyle\frac{\lambda}{2} (\tau\!-\!1) \big( (\tau+1) u_{e2} +\rho (\tau\!-\!1) r_e^2 \big) r_e^2} \, e^{25} 
+{\scriptstyle(\tau\!-\!1) \big( \rho (\tau\!-\!1) + \frac{\lambda}{2} (\tau+1) \big) u_{e1} r_e^2} \, e^{26} \\
&
-{\scriptstyle\frac{\lambda}{2\Delta_e} (\tau\!-\!1) (\tau + 1) \big( \lambda u_{e1} - y r_e^2  \big) r_e^4 }\, e^{35} \\
&
-{\scriptstyle\frac{1}{2\Delta_e} (\tau\!-\!1)  \big[ 2\rho(\tau\!-\!1)|u_e|^2 - (\tau+1) (2s_e^2 u_{e2} - \lambda |u_e|^2)  } \\
&\quad\quad\quad\quad\quad
{\scriptstyle
- \big( 2 \rho (\tau\!-\!1) s_e^2 - (\tau+1) (2 \lambda s_e^2 \!-\! \lambda^2 u_{e2} \!-\! 2 x u_{e2} \!-\! 2 y u_{e1}) \big) r_e^2     
+ \lambda x (\tau+1) r_e^4 \big] r_e^2 }\, e^{36} \\
&
-{\scriptstyle\frac{\lambda}{2\Delta_e} (\tau\!-\!1) (\tau + 1) \big( |u_e|^2 -\lambda u_{e2} r_e^2 + x r_e^4 \big) r_e^2} \, e^{45} \\
&
-{\scriptstyle\frac{1}{2\Delta_e} (\tau\!-\!1) (\tau + 1) \big( 2u_{e1} s_e^2 -(\lambda^2 u_{e1} -2xu_{e1}+2 y u_{e2}) r_e^2 + \lambda y r_e^4 \big) r_e^2} \, e^{46} 
\,,
\end{aligned}
$$

$$
\begin{aligned}
%%% 3,4 
\tfrac{2r_e^4 \Delta_e^2}{k_e^2} (\Omega^{\tau})^3_4  & =
-{\scriptstyle\big( (\tau^2 \!-\! 2 \tau +5) |u_e|^2 + 4 \lambda (\tau\!-\!1) u_{e2} r_e^2 
- (\tau\!-\!1) (\rho (\tau\!-\!1) + 4 x) r_e^4 \big) }\, e^{12}\\
&
+{\scriptstyle\frac{1}{\Delta_e} \big[ 
(\tau^2 \!-\! 2 \tau +5) |u_e|^2 u_{e1}
+ 4 \lambda (\tau\!-\!1) u_{e1} u_{e2} r_e^2 } \\
&\quad\quad\quad
{\scriptstyle
- \big( 2 \lambda^2 (\tau\!-\!1) u_{e1}
+ \rho \lambda (\tau\!-\!1) (\tau+3) u_{e1}
 - x (\tau^2 \!-\! 2 \tau +5) u_{e1}  + y (\tau+1)^2 u_{e2} \big) r_e^4 
+ y (\tau\!-\!1) (\rho (\tau+3)+2 \lambda) r_e^6
\big]}\, e^{13}\\
&
+{\scriptstyle\frac{1}{\Delta_e} \big[ 
(\tau^2 \!-\! 2 \tau +5) |u_e|^2 u_{e2}
- \big( \rho (\tau\!-\!1)(\tau+3) |u_e|^2 + \lambda (\tau^2+3) u_{e1}^2 
+ \lambda (\tau^2 \!-\! 4 \tau +7) u_{e2}^2 \big) r_e^2 } \\
&\quad\quad\quad
{\scriptstyle
- \big( 2 \lambda^2 (\tau\!-\!1) u_{e2} - \rho\lambda (\tau\!-\!1) (\tau+3) u_{e2}
- x (\tau^2 \!-\! 2 \tau +5) u_{e2} - y (\tau+1)^2 u_{e1} \big) r_e^4 
- x (\tau\!-\!1) (\rho (\tau+3) - 2 \lambda) r_e^6
\big]}\, e^{14}\\
&
-{\scriptstyle\frac{1}{\Delta_e} \big[ 
(\tau^2 \!-\! 2 \tau +5) |u_e|^2 u_{e2} 
+ \big( \rho (\tau\!-\!1)(\tau+3) |u_e|^2 - \lambda (\tau^2+3) u_{e1}^2 
- \lambda (\tau^2 \!-\! 4 \tau +7) u_{e2}^2 \big) r_e^2 } \\
&\quad\quad\quad
{\scriptstyle
- \big( 2 \lambda^2 (\tau\!-\!1) u_{e2} + \rho\lambda (\tau\!-\!1) (\tau+3) u_{e2}
- x (\tau^2 \!-\! 2 \tau +5) u_{e2} - y (\tau+1)^2 u_{e1} \big) r_e^4 
+ x (\tau\!-\!1) (\rho (\tau+3) + 2 \lambda) r_e^6
\big]}\, e^{23}\\
&
+{\scriptstyle\frac{1}{\Delta_e} \big[ 
(\tau^2 \!-\! 2 \tau +5) |u_e|^2 u_{e1} 
+ 4 \lambda (\tau\!-\!1) u_{e1} u_{e2} r_e^2 } \\
&\quad\quad\quad
{\scriptstyle
- \big( 2 \lambda^2 (\tau\!-\!1) u_{e1}
- \rho \lambda (\tau\!-\!1) (\tau+3) u_{e1}
 - x (\tau^2 \!-\! 2 \tau +5) u_{e1}  + y (\tau+1)^2 u_{e2} \big) r_e^4 
- y (\tau\!-\!1) (\rho (\tau+3)-2 \lambda) r_e^6
\big]}\, e^{24}\\
&
-{\scriptstyle\frac{\tau^2 \!-\! 2 \tau +5}{\Delta_e^2} \big[ 
|u_e|^4 - 2\lambda |u_e|^2 u_{e2} r_e^2 + (2x+\lambda^2) |u_e|^2 r_e^4 
- 2 \lambda (x u_{e2} +y u_{e1}) r_e^6 
+(x^2 + y^2) r_e^8 \big]}\, e^{34} \\
&
{\scriptstyle
\ +\ \lambda (\tau\!-\!1)^2 (\lambda r_e^2 - 2 u_{e2})\, r_e^2} \, e^{56}
\,,\\[6pt]
\end{aligned}
$$

$$
\begin{aligned}
%%% 3,5 ...........
\tfrac{2r_e^4 \Delta_e^2}{k_e^2} (\Omega^{\tau})^3_5  & = 
-{\scriptstyle\frac{1}{2\Delta_e} (\tau\!-\!1) (\tau + 1) \big( \lambda |u_{e}|^2 + (\lambda s_e^2 -2 y u_{e1}) r_e^2 \big) r_e^2} \, e^{15} 
+{\scriptstyle\frac{1}{\Delta_e} (\tau\!-\!1) (\tau + 1) \big( \lambda u_{e2} - s_e^2  - x r_e^2 \big) u_{e1} r_e^2} \, e^{16} \\
&
+{\scriptstyle\frac{1}{\Delta_e} (\tau\!-\!1) (\lambda u_{e1} - y r_e^2) \big( (\tau+1) u_{e2} + \rho (\tau\!-\!1) r_e^2  \big) r_e^2} \, e^{25} \\
&
+{\scriptstyle\frac{1}{2\Delta_e} (\tau\!-\!1) 
\big[ 
2 \rho(\tau\!-\!1) |u_{e}|^2 + \lambda(\tau+1)(u_{e1}^2 - u_{e2}^2)+2(\tau+1) s_{e}^2 u_{e2} } \\
&\quad\quad\quad\quad\quad
{\scriptstyle
-\big( 
 2 \rho\lambda (\tau\!-\!1) u_{e2} + \lambda (\tau+1)s_{e}^2 - 2 x (\tau+1) u_{e2}
\big) r_e^2 
+ 2 \rho x (\tau\!-\!1) r_e^4
\big] r_e^2} \, e^{26} \\
&
-{\scriptstyle\frac{1}{2\Delta_e^2} (\tau\!-\!1) 
\big[ 
\lambda (\tau+1) |u_{e}|^2 u_{e2} 
+\big( 
  \rho\lambda(\tau\!-\!1) |u_{e}|^2 + \lambda^2(\tau+1) (u_{e1}^2 - u_{e2}^2) 
-\! \lambda(\tau\!+\!1) s_{e}^2 u_{e2} 
\big) r_e^2  } \\
&\quad\quad\quad\quad\quad
{\scriptstyle
-\!\big( 
\rho\lambda s_{e}^2 (\tau\!-\!1) \!-\! \lambda(\tau\!+\!1) (\lambda s_{e}^2 \!-\! 4 y u_{e1})
\big) r_e^4 
+2 y^2 (\tau\!+\!1) r_e^6
\big] r_e^2} \, e^{35}\\
&
-{\scriptstyle\frac{1}{2\Delta_e^2} (\tau\!-\!1) 
\big[ 
(2\rho (\tau\!-\!1) + \lambda (\tau+1)) |u_{e}|^2 u_{e1}
- \big( 2 \rho (\tau\!-\!1) s_{e}^2 u_{e1} - (\tau+1) (\lambda s_{e}^2 u_{e1}
- 2 \lambda^2 u_{e1} u_{e2} - 2 y |u_{e}|^2)
\big) r_e^2  } \\
&\quad\quad\quad\quad\quad
{\scriptstyle
+ 2 \lambda (\tau+1) (x u_{e1} + y u_{e2}) r_e^4 - 2 xy (\tau+1) r_e^6
\big] r_e^2} \, e^{36} \\
&
-{\scriptstyle\frac{1}{2\Delta_e^2} (\tau\!-\!1) (\tau + 1) 
\big[ 
\lambda |u_{e}|^2 u_{e1} + 
(\lambda s_{e}^2 u_{e1} - 2 \lambda^2 u_{e1} u_{e2} - 2 y |u_{e}|^2) r_e^2
+2 \lambda (x u_{e1} + y u_{e2}) r_e^4 - 2 x y r_e^6
\big] r_e^2} \, e^{45} \\
&
-{\scriptstyle\frac{1}{2\Delta_e^2} (\tau\!-\!1) 
\big[ 
\big( 2 \rho (\tau\!-\!1) u_{e2} -(\tau+1)( \lambda u_{e2} - 2 s_{e}^2) \big)|u_{e}|^2 + \big( \rho\lambda (\tau\!-\!1) s_{e}^2 + \lambda(\tau+1) (\lambda s_{e}^2 - 4 x u_{e2}) \big) r_e^4  +  2 x^2 (\tau+1)  r_e^6 
}
 \\
&\quad\quad\quad\quad\quad
{\scriptstyle-\big( \rho (\tau\!-\!1) (\lambda |u_{e}|^2 + 2 s_{e}^2 u_{e2}) 
     + (\tau+1) (3 \lambda s_e^2 u_{e2} + \lambda^2 (u_{e1}^2 - u_{e2}^2) 
     - 4 x |u_{e}|^2) \big) r_e^2}\big] r_e^2\, e^{46} 
\,,\\[6pt]
\end{aligned}
$$

$$
\begin{aligned}
%%% 3,6 ...........
\tfrac{2r_e^4 \Delta_e^2}{k_e^2} (\Omega^{\tau})^3_6  & =  
\ \,{\scriptstyle \frac{1}{\Delta_e} (\tau\!-\!1) 
(\lambda u_{e1} - y r_e^2) 
\big( (\tau+1) u_{e2} - \rho (\tau\!-\!1) r_e^2  \big) r_e^2} \, e^{15} \\
&
-{\scriptstyle\frac{1}{2\Delta_e} (\tau\!-\!1) 
\big[
2\rho(\tau\!-\!1) |u_{e}|^2 - (\tau+1) (\lambda (u_{e1}^2 - u_{e2}^2) + 2 s_{e}^2 u_{e2}) } \\
&\quad\quad\quad\quad\quad
{\scriptstyle
- \big(  2\rho\lambda (\tau\!-\!1) u_{e2}
-\lambda (\tau+1) s_{e}^2 + 2 x (\tau+1) u_{e2}
\big) r_e^2 
+ 2 \rho x (\tau\!-\!1)  r_e^4
\big] r_e^2} \, e^{16} \\
&
+{\scriptstyle\frac{1}{2\Delta_e} (\tau\!-\!1) (\tau+1)
\big[
\lambda (u_{e1}^2 - u_{e2}^2) +   (\lambda s_{e}^2 - 2 y u_{e1}) r_e^2 
\big] r_e^2} \, e^{25}
-{\scriptstyle\frac{1}{\Delta_e} (\tau\!-\!1) (\tau+1)
\big[
\lambda u_{e2} - s_e^2 - x r_e^2  
\big] u_{e1}\, r_e^2} \, e^{26}\\
&
-{\scriptstyle\frac{1}{2\Delta_e^2} (\tau\!-\!1) 
\big[
\big(2 (\tau+1) s_e^2 - (2 \rho (\tau\!-\!1) + \lambda (\tau+1)) u_{e2}\big) |u_{e}|^2
-  \big( \rho\lambda (\tau\!-\!1) s_e^2 - \lambda (\tau+1)(\lambda s_e^2 - 4 x u_{e2}) \big) r_e^4 + 2 x^2 (\tau+1) r_e^6 
} \\
&\quad\quad\quad\quad\quad
{\scriptstyle+ \big(\rho (\tau\!-\!1) (2 s_e^2 u_{e2} +\lambda |u_{e}|^2) -\! (\tau+1) (3 \lambda s_e^2 u_{e2} + \lambda^2 (u_{e1}^2 \!-\! u_{e2}^2) \!-\! 4 x |u_{e}|^2)\big) r_e^2\big] r_e^2}
 \, e^{36}\\
&
+{\scriptstyle\frac{1}{2\Delta_e^2} (\tau\!-\!1) 
\big[
\lambda (\tau+1) u_{e2} |u_{e}|^2 
-\!\big(\rho\lambda (\tau\!-\!1) |u_{e}|^2 -\! \lambda(\tau+1) (\lambda (u_{e1}^2 \!-\! u_{e2}^2) \!-\! s_e^2 u_{e2})
\big) r_e^2} \\
&\quad\quad\quad\quad\quad
{\scriptstyle
+\big(\rho\lambda (\tau\!-\!1) s_e^2 + \lambda(\tau+1) (\lambda s_e^2 - 4 y u_{e1})
\big) r_e^4 + 2 y^2 (\tau+1) r_e^6
\big] r_e^2} \, e^{45}\\
&
-{\scriptstyle\frac{1}{2\Delta_e^2} (\tau\!-\!1) 
\big[
 (2 \rho (\tau\!-\!1) - \lambda (\tau+1)) u_{e1} |u_{e}|^2-\!2 \lambda(\tau\!+\!1) (x u_{e1} + y u_{e2}) r_e^4 + 2xy(\tau\!+\!1)  r_e^6
}\\
&\quad\quad\quad\quad\quad
{\scriptstyle
-\big(2 \rho (\tau\!-\!1)  s_e^2 u_{e1} + \lambda(\tau+1) s_e^2 u_{e1}
-\!2(\tau\!+\!1) (\lambda^2 u_{e1} u_{e2} +y|u_{e}|^2)
\big) r_e^2\big] r_e^2
} \, e^{46}
\,,
\end{aligned}
$$

$$
\begin{aligned}
%%% 5,6 ...........
\tfrac{2r_e^4 \Delta_e^2}{k_e^2} (\Omega^{\tau})^5_6  & =
-{\scriptstyle \tfrac{2r_e^4 \Delta_e^2}{k_e^2}} \big[\, (\Omega^{\tau})^1_2 + (\Omega^{\tau})^3_4 \,\big] 
- {\scriptstyle 4 (\tau\!-\!1) (\lambda u_{e2} - s_e^2 - x r_e^2) r_e^2}  \, e^{12} \\
&
+{\scriptstyle\tfrac{2}{\Delta_e} (\tau\!-\!1) \big( 2 u_{e1} (\lambda u_{e2} - s_e^2) 
- (\rho \lambda + \lambda^2 + 2 x) u_{e1} r_e^2 
+  (\rho + \lambda) y r_e^4 \big) 
r_e^2}  \, e^{13}\\
&
+{\scriptstyle\tfrac{2}{\Delta_e} (\tau\!-\!1) \big( 2 u_{e2} + (\rho - \lambda) r_e^2 \big)
\big( \lambda u_{e2} - s_e^2 - x r_e^2 \big) 
r_e^2}  \, e^{14} \\
&
-{\scriptstyle\tfrac{2}{\Delta_e} (\tau\!-\!1) \big( 2 u_{e2} - (\rho + \lambda) r_e^2 \big)
\big( \lambda u_{e2} - s_e^2 - x r_e^2 \big) 
r_e^2}  \, e^{23} \\
&
+{\scriptstyle\tfrac{2}{\Delta_e} (\tau\!-\!1) \big( 2 u_{e1} (\lambda u_{e2} - s_e^2) 
+ (\rho \lambda - \lambda^2 - 2 x) u_{e1} r_e^2 
-  (\rho - \lambda) y r_e^4 \big) 
r_e^2}  \, e^{24} \\
&
-{\scriptstyle\tfrac{4}{\Delta_e^2} (\tau\!-\!1) \big[ (\lambda u_{e2}-s_e^2) |u_e|^2 
+ \big( \lambda u_{e2} s_e^2 - \lambda^2 |u_e|^2 - x |u_e|^2 \big) r_e^2 
+ \big( 2 \lambda (x u_{e2} + y u_{e1}) - x s_e^2 \big) r_e^4 
- (x^2+y^2) r_e^6 
\big] 
r_e^2}  \, e^{34} 
\,.\\[6pt]
\end{aligned}
$$

%%%\appendix
\section{Appendix B}\label{apendiceB}

\noindent 
This Appendix is devoted to the computation of ${\rm Tr}(A^\kappa\wedge A^\kappa)$ for a Gauduchon connection $\nabla^\kappa$ on the holomorphic tangent bundle $T^{1,0}G$. In particular, the proof of Lemma \ref{lemm_ass} will follow. \medskip

Let $(G,J)$ be a Lie group equipped with a left-invariant complex structure. Let $\{\zeta^1,\zeta^2,\zeta^3\}$ be a left-invariant (1,0)-coframe satisfying \eqref{J-nilp-lambda=0}. Let also $\omega$ and $H$ be two left-invariant $J$-Hermitian metrics on $G$ given by
\begin{equation}\label{F-nilp2}
\omega  =  \frac i2 ( r^2\, \zeta^{1\bar{1}} + s^2\, \zeta^{2\bar{2}} + k^2\, \zeta^{3\bar{3}}) \qquad \text{and} \qquad H =  \frac i2 ( {\tilde r}^2\, \zeta^{1\bar{1}} + {\tilde s}^2\, \zeta^{2\bar{2}} + {\tilde k}^2\, \zeta^{3\bar{3}})\,,
\end{equation}
for some $r,s,k,{\tilde r}, {\tilde s}, {\tilde k}\in\RR^\ast$. If $\nabla^\kappa$ is a Gauduchon connection of $H$ and $A^\kappa$ its curvature form, to compute the trace ${\rm Tr}(A^\kappa \wedge A^\kappa)$ by using \eqref{curvature} and \eqref{traceRmRm} we need to write the connection 1-forms ${\sigma^\kappa}$ in terms of an adapted basis $\{e^l\}_{l=1}^6$ for $\omega$ (see Proposition \ref{adapted-ecus}). In the following, we will denote by $(\sigma^\kappa)^{i}_{ j}$ and $(A^\kappa)^i_j$ the connection 1-forms and the curvature 2-forms, respectively, written in terms of $\{e^l\}_{l=1}^6$.\smallskip

Let $\{{\tilde e}^{l}\}_{{l}=1}^6$ be an {\em adapted basis} for the metric $H$ and $\{\tilde e_l\}_{l=1}^6$ its dual. In view of Section~\ref{sub-trace}, the connection 1-forms $(\sigma^\kappa)^{\tilde i}_{\tilde j}$ associated to $\nabla^\kappa$ are given by
$$
\nabla_{\tilde e_k} \tilde e_j = (\sigma^\kappa)^{\tilde 1}_{\tilde j}(\tilde e_k)\, 	\tilde e_1 +\cdots+ (\sigma^\kappa)^{\tilde 6}_{\tilde j}(\tilde e_k)\, \tilde e_6\,.
$$
On the other hand, if $\{e_l\}_{l=1}^6$ denotes the dual basis of $\{e^l\}_{l=1}^6$, and $M:=(M^{ p}_j)$ is the change-of-basis matrix from $\{e_l\}$ to $\{\tilde e_l\}$, i.e.
$$
\tilde e_j= M^{ p}_j \, { e}_{ p}\,, \quad \text{for every}\,\, 1\leq j\leq 6\,,
$$
then one gets
$$
\nabla_{\tilde e_k}\tilde e_j  =  \nabla_{M^p_k \, {e}_{p}} ( M^{q}_j \, { e}_{ q})=  M^{ p}_k \, M^{ q}_j \, \nabla_{{ e}_{ p}} { e}_{ q}= M^{ p}_k \, M^{ q}_j \, (\sigma^\kappa)^{ l}_{ q}({ e}_{ p})\, { e}_{ l}= M^{ p}_k \, M^{ q}_j\, N^i_{ l} \, (\sigma^\kappa)^{ l}_{ q}({ e}_{ p}) \tilde e_i\,,
$$
with $N:=M^{-1}$ (that is, ${ e}_{ l}= N^i_{ l}\,\tilde e_i$), and hence
\begin{equation}\label{rel-sigmas}
(\sigma^\kappa)^{\tilde i}_{\tilde j}(\tilde e_k) = {\tilde g}(\nabla_{\tilde e_k} \tilde e_j,\tilde e_i) =  M^{ p}_k \, M^{ q}_j\, N^i_{ l} \, (\sigma^\kappa)^{ l}_{ q}({ e}_{ p})\,.
\end{equation}
Since the $(1,0)$-coframe $\{\zeta^1,\zeta^2,\zeta^3\}$ only depends on the complex structure $J$, by means of \eqref{J-nilp-sigma-anterior}, \eqref{J-taus} and \eqref{es-taus} we have
\begin{equation*}\label{rel-1}\begin{aligned}
e^1+i\,e^2 = r\, \zeta^1\,, \quad\quad
{\tilde r}\, \zeta^1= {\tilde e}^{ 1}+i\,{\tilde e}^{ 2}\,,\\
e^3+i\,e^4 = s\,\zeta^2\,, \quad\quad 
{\tilde s}\,\zeta^2 = {\tilde e}^{ 3}+i\,{\tilde e}^{ 4}\,,\\
e^5+i\,e^6 = k\, \zeta^3\,, \quad\quad 
{\tilde k}\, \zeta^3= {\tilde e}^{ 5}+i\,{\tilde e}^{ 6}\,,
\end{aligned}\end{equation*}
which directly implies
$$
\tilde e^1 = \frac{\tilde r}{ r}\, { e}^{ 1}\,,\quad\quad 
\tilde e^2 = \frac{\tilde r}{r}\, { e}^{ 2}\,, \quad\quad
\tilde e^3 = \frac{ \tilde s } {  s } \, { e}^{ 3}\,, \quad\quad
\tilde e^4 = \frac{ \tilde s } {  s } \, { e}^{ 4}\,,\quad\quad
\tilde e^5=\frac{\tilde k}{ k}\, { e}^{5}\,,\quad\quad
\tilde e^6=\frac{\tilde k}{ k}\, { e}^{6}\,.
$$
Thereby, the change-of-basis matrix $M$ from $\{e_l\}$ to $\{\tilde e_l\}$ is given by the diagonal matrix
$$
M:= {\rm diag}\left( \frac{r}{\tilde r}\,,\frac{r}{\tilde r}\,,\frac{s}{\tilde s}\,,\frac{s}{\tilde s}\,,\frac{k}{\tilde k}\,,\frac{k}{\tilde k}\right)\,.
$$
Thus, by means of \eqref{rel-sigmas}, one gets
\begin{equation}\label{relation_1-forms}
(\sigma^{\kappa})^i_j(e_k) =  M^{i}_i\, N^j_{j}\, N^k_{k} \, (\sigma^{\kappa})^{\tilde i}_{\tilde j}({\tilde e}_{k})\,,
\end{equation}
or, equivalently, 
$$
(\sigma^{\kappa})^i_j =  M^{i}_i\, N^j_{j}\, N^k_{k} \, (\sigma^{\kappa})^{\tilde i}_{\tilde j}({\tilde e}_{k})\, e^k\,.
$$
Finally, since the connection 1-forms $(\sigma^\kappa)^{\tilde i}_{\tilde j}$ are given in the proof of Proposition \ref{traza-tau}, 
a direct computation by means of \eqref{relation_1-forms} yields that 
\begin{equation}\label{conn-form-app}
\begin{aligned}
(\sigma^{\kappa})^1_2=& -\tfrac{\tilde k^2}{k\,{\tilde r}^2}{\scriptstyle(\kappa-1)}\, e^{6}\,,\\[3pt]
(\sigma^{\kappa})^1_5=& -\tfrac{{\tilde k}^2}{2\,k\,{\tilde r}^2}{\scriptstyle (\kappa+1)}\, e^{1} + \tfrac{r\,{\tilde k}^2}{2\, s\,k\,{\tilde r}^2}\, {\scriptstyle \rho (\kappa - 1)}\, e^{3}\,,\\[3pt]
(\sigma^{\kappa})^1_6=&\ \tfrac{{\tilde k}^2}{2\,k\,{\tilde r}^2}{\scriptstyle (\kappa+1)}\, e^{2} + \tfrac{r\,{\tilde k}^2}{2 \, s\,k\,{\tilde r}^2}\, {\scriptstyle \rho(\kappa-1) }\, e^{4}\,,\\[3pt]
(\sigma^{\kappa})^3_4=&\ \tfrac{{\tilde k}^2}{k\,{\tilde s}^2}\, {\scriptstyle y(\kappa-1)}\, e^{5} -\tfrac{{\tilde k}^2}{k\,{\tilde s}^2}\, {\scriptstyle x(\kappa-1)}\, e^{6}\,,\\[3pt]
(\sigma^{\kappa})^3_5=& -\tfrac{s\,{\tilde k}^2}{2\,r\,k\,{\tilde s}^2} \, {\scriptstyle \rho (\kappa-1)}\, e^{1} -\tfrac{{\tilde k}^2}{2\, k\,{\tilde s}^2}\, {\scriptstyle x(\kappa+1)}\, e^{3} -\tfrac{{\tilde k}^2}{2\,k\,{\tilde s}^2}\,{\scriptstyle y(\kappa+1)}\, e^{4}\,,\\[3pt]
(\sigma^{\kappa})^3_6=&-\tfrac{s\,{\tilde k}^2}{2\,r\,k\,{\tilde s}^2}\, {\scriptstyle \rho(\kappa-1)}\, e^{2} -\tfrac{{\tilde k}^2}{2\,k\,{\tilde s}^2}\, {\scriptstyle y(\kappa+1)}\, e^{3}+\tfrac{{\tilde k}^2}{2\,k\,{\tilde s}^2}\,{\scriptstyle x(\kappa+1)}\, e^{4}\,,\\[4pt]
(\sigma^{\kappa})^1_3=&(\sigma^{\kappa})^1_4=(\sigma^{\kappa})^2_3
=(\sigma^{\kappa})^2_4=(\sigma^{\kappa})^5_6=0\,,
\end{aligned}
\end{equation}
together with the following relations
$$
(\sigma^{\kappa})^2_5 = - (\sigma^{\kappa})^1_6\,,\quad (\sigma^{\kappa})^2_6 = (\sigma^{\kappa})^1_5\,,\quad(\sigma^{\kappa})^4_5 = - (\sigma^{\kappa})^3_6\,, \quad (\sigma^{\kappa})^4_6 = (\sigma^{\kappa})^3_5\,,
$$
and $(\sigma^{\kappa})^i_j = - (\sigma^{\kappa})^j_i$.

\begin{lemma}\label{lemm_app}
Let $G$ be a $2$-step nilpotent Lie group equipped with a left-invariant complex structure $J$ which admits a left-invariant $(1,0)$-coframe $\{ \zeta^l\}_{l=1}^3$ satisfying \eqref{J-nilp-lambda=0}. Let $\omega$ and $H$ be two left-invariant $J$-Hermitian metrics defined by \eqref{F-nilp2}. Then, for any Gauduchon connection $\nabla^{\kappa}$ associated to $H$, the trace of its curvature satisfies 
$$
{\rm Tr}(A^\kappa\wedge A^\kappa)=C\zeta^{12\bar 1\bar 2}\,,
$$
where $C=C(\rho,x,y;\omega,H;\kappa)$ is a constant depending both on the Hermitian structures and the connection. More precisely, we have
$$
\begin{aligned}
{\rm Tr}\, (A^\kappa\wedge A^\kappa)= \frac{(\kappa-1) \,\tilde k^4}{2k^2\tilde r^6\tilde s^6}&\Big\lbrace \rho(\kappa-1) \Big[ (2\kappa\, r^2\tilde k^2 + k^2\tilde r^2)\tilde s^6 + (x^2+y^2)(2\kappa\, s^2\tilde k^2 + k^2\tilde s^2)\tilde r^6 \Big] \\
&+ 4x(\kappa-1) \Big( (x^2+y^2)\tilde r^4+\tilde s^4 \Big) k^2\tilde r^2\tilde s^2\\
&-x(\kappa+1)^2 \Big( (x^2+y^2)s^2\tilde r^6 + r^2\tilde s^6 \Big) \tilde k^2  \Big\rbrace\, \zeta^{12\bar 1\bar 2}.
\end{aligned}
$$
\end{lemma}

\begin{proof} Let $(\sigma^{\kappa})^{ i}_{ j}$ be the connection 1-forms of the Gauduchon connection $\nabla^{\kappa}$ given in \eqref{conn-form-app}. By means of \eqref{J-nilp-real-basis} and \eqref{curvature}, a direct computations yields that the curvature 2-forms $(A^\kappa)^{ i}_{ j}$ of $\nabla^\kappa$ are
$$
\begin{aligned}
(A^\kappa)^{1}_{2}=&\ \tfrac{4 (\kappa-1) k^2\tilde r^2\tilde k^2 -(\kappa+1)^2r^2\tilde k^4}{2r^2k^2\tilde r^4}\, e^{12}
-\tfrac{\rho(\kappa-1) \left( (\kappa+1)r^2\tilde k^2 + 2k^2\tilde r^2 \right) \tilde k^2}{2rsk^2\tilde r^4}(e^{14}+e^{23})\\
&+\tfrac{(\kappa-1) \left( \rho (\kappa-1)r^2\tilde k^2+4xk^2\tilde r^2 \right) \tilde k^2}{2s^2k^2\tilde r^4}\, e^{34}\,,\\[3pt]
(A^\kappa)^{1}_{3}=&-\tfrac{ \left( \rho(\kappa-1)^2+x(\kappa+1)^2 \right) \tilde k^4}{4k^2\tilde r^2\tilde s^2}(e^{13}+e^{24}) - \tfrac{y(\kappa+1)^2 \tilde k^4}{4k^2\tilde r^2\tilde s^2}(e^{14}-e^{23})\,,\\[3pt]
(A^\kappa)^{1}_{4}=&\ \tfrac{y(\kappa+1)^2 \tilde k^4}{4k^2\tilde r^2\tilde s^2}(e^{13}+e^{24})-\tfrac{ \left( \rho(\kappa-1)^2+x(\kappa+1)^2 \right) \tilde k^4}{4k^2\tilde r^2\tilde s^2}(e^{14}-e^{23})\,,\\[3pt]
(A^\kappa)^{1}_{5}=&-\tfrac{(\kappa-1)(\kappa+1)\tilde k^4}{2k^2\tilde r^4}\, e^{26} - \tfrac{\rho(\kappa-1)^2 r \tilde k^4}{2s k^2 \tilde r^4}\, e^{46}\,,\\[3pt]
(A^\kappa)^{1}_{6}=&-\tfrac{(\kappa-1)(\kappa+1)\tilde k^4}{2k^2\tilde r^4}\, e^{16} + \tfrac{\rho(\kappa-1)^2 r \tilde k^4}{2s k^2 \tilde r^4}\, e^{36}\,,\\[3pt]
(A^\kappa)^{3}_{4}=&\ \tfrac{(\kappa-1) \left( \rho (\kappa-1)s^2\tilde k^2+4xk^2\tilde s^2 \right) \tilde k^2}{2r^2k^2\tilde s^4}\, e^{12} +\tfrac{\rho\,y(\kappa-1) \left( (\kappa+1)s^2\tilde k^2 + 2k^2\tilde s^2 \right) \tilde k^2}{2rsk^2\tilde s^4}(e^{13}-e^{24})\\
&-\tfrac{\rho\,x(\kappa-1) \left( (\kappa+1)s^2\tilde k^2 + 2k^2\tilde s^2 \right) \tilde k^2}{2rsk^2\tilde s^4}(e^{14}+e^{23}) + \tfrac{(x^2+y^2)\left(  4(\kappa-1)k^2\tilde s^2 -(\kappa+1)^2 s^2\tilde k^2 \right) \tilde k^2}{2s^2k^2\tilde s^4}\, e^{34}\,,\\[3pt]
(A^\kappa)^{3}_{5}=& -\tfrac{\rho(\kappa-1)^2 s \tilde k^4}{2r k^2\tilde s^4}(y\,e^{25}-x\,e^{26}) -\tfrac{(\kappa-1)(\kappa+1)\tilde k^4}{2k^2\tilde s^4} \left( y^2\,e^{35}-xy (e^{36}+e^{45})+x^2\,e^{46} \right)\,,\\[3pt]
(A^\kappa)^{3}_{6}=&\ \tfrac{\rho(\kappa-1)^2 s \tilde k^4}{2r k^2\tilde s^4}(y\,e^{15}-x\,e^{16}) +\tfrac{(\kappa-1)(\kappa+1)\tilde k^4}{2k^2\tilde s^4} \left( xy (e^{25}-e^{36})-x^2\,e^{26}+ y^2\,e^{35} \right)\,,\\[3pt]
(A^\kappa)^{5}_{6}=& -\tfrac{\left( \rho(\kappa-1)^2s^2\tilde r^4-(\kappa+1)^2 r^2 \tilde s^4\right) \tilde k^4}{2r^2k^2\tilde r^4\tilde s^4}\, e^{12} - \tfrac{\rho y(\kappa-1)(\kappa+1)s\tilde k^4}{2rk^2\tilde s^4} (e^{13}-e^{24})\\ 
&+ \tfrac{\rho (\kappa-1)(\kappa+1)(r^2\tilde s^4 + x\,s^2\tilde r^4)\tilde k^4}{2rsk^2\tilde r^4\tilde s^4} (e^{14}+e^{23})
-\tfrac{\left( \rho(\kappa-1)^2r^2\tilde s^4 - (x^2+y^2)(\kappa+1)^2 s^2\tilde r^4 \right) \tilde k^4}{2s^2k^2\tilde r^4\tilde s^4}\, e^{34}\,,
\end{aligned}
$$
together with the following relations
$$
\begin{aligned}
&(A^\kappa)^{2}_{3}=-(A^\kappa)^{1}_{4}\,,\quad (A^\kappa)^{2}_{4}=(A^\kappa)^{1}_{3}\,,\quad (A^\kappa)^{2}_{5}=-(A^\kappa)^{1}_{6}\,,\quad (A^\kappa)^{2}_{6}=(A^\kappa)^{1}_{5}\,,\\
&(A^\kappa)^{4}_{5}=-(A^\kappa)^{3}_{6}\,,\quad(A^\kappa)^{4}_{6}=(A^\kappa)^{3}_{5}\,.
\end{aligned}
$$
Therefore, the claim follows by using \eqref{traceRmRm} and \eqref{relacion}.
\end{proof}

\end{document}